\documentclass{amsart}
\usepackage{color}
\usepackage{graphicx}
\usepackage{hyperref}
\usepackage{amssymb}
\usepackage{amsfonts}
\usepackage{euscript}
\usepackage[all]{xy}
\usepackage{bbm}

\numberwithin{equation}{section}
\renewcommand{\subsection}[1]{\hspace{-\parindent}\refstepcounter{subsection}{\bf (\arabic{section}\alph{subsection}) #1.}\addcontentsline{toc}{subsection}{\bf #1.}}

\newenvironment{nouppercase}{%
  \renewcommand{\uppercasenonmath}[1]{}}{}

\theoremstyle{plain}
\newtheorem{thm}{Theorem}[section]

\newtheorem{addendum}[thm]{Addendum}
\newtheorem{cor}[thm]{Corollary}
\newtheorem{corollary}[thm]{Corollary}

\newtheorem{prop}[thm]{Proposition}

\newtheorem{remark}[thm]{Remark}

\newtheorem{example}[thm]{Example}

\newtheorem{lemma}[thm]{Lemma}

\newtheorem*{claim*}{Claim} 
\newtheorem*{lemma*}{Lemma}
\newtheorem*{theorem*}{Theorem}
\newtheorem*{conjecture*}{Conjecture}

\newcommand{\bC}{{\mathbb C}}

\newcommand{\bK}{{\mathbb K}}

\newcommand{\bR}{{\mathbb R}}

\newcommand{\bZ}{{\mathbb Z}}

\newcommand{\scrA}{\EuScript A}

\newcommand{\scrC}{\EuScript C}

\newcommand{\scrF}{\EuScript F}
\newcommand{\scrG}{\EuScript G}
\newcommand{\scrH}{\EuScript H}
\newcommand{\scrI}{\EuScript I}
\newcommand{\scrJ}{\EuScript J}

\newcommand{\scrP}{\EuScript P}
\newcommand{\scrQ}{\EuScript Q}

\newcommand{\scrS}{\EuScript S}

\newcommand{\frakg}{\mathfrak{g}}

\newcommand{\half}{{\textstyle\frac{1}{2}}}

\newcommand{\iso}{\cong}
\newcommand{\htp}{\simeq}
\newcommand{\smooth}{C^\infty}

\newcommand{\aff}{\mathit{aff}}
\newcommand{\Id}{\mathbbm{1}}

\title[LEFSCHETZ FIBRATIONS]{\Large\larger\rm Fukaya $A_\infty$-structures associated to\\ Lefschetz fibrations. IV 1/2}
\author{Paul Seidel}

\begin{document}
\begin{nouppercase}
\maketitle
\end{nouppercase}

\begin{abstract}
We describe a construction of the Fukaya category of an exact symplectic Lefschetz fibration, together with its closed-open string map. 
\end{abstract}

\section{Introduction}

\subsection{Context}
The surrounding mathematical landscape can surveyed as follows (this is an idealized description of ``things as they should be'': it glosses over the fact that several approaches exist, differing in the details and limitations). Take a symplectic Lefschetz fibration
\begin{equation} \label{eq:lefschetz}
\pi: E^{2n} \longrightarrow B,
\end{equation}
where $B \iso \bR^2$. There is always a distinguished Hamiltonian automorphism $\sigma$ of $E$, which is obtained by rotating the base by $2\pi$ near infinity. The ``closed string'' structure of interest is a sequence of Floer cohomology groups $\mathit{HF}^*(E,r)$, which are roughly speaking the fixed point Floer cohomology groups of the iterates $\sigma^r$, $r \in \bZ$. Their definition (involving a suitable choice of perturbation, to get rid of the fixed points at infinity) is arranged so that $\mathit{HF}^*(E,0)$ is the ``vanishing cohomology'' of \eqref{eq:lefschetz} (the cohomology relative to a fibre at infinity); and all $\mathit{HF}^*(E,r)$ vanish if the fibration is trivial (has no singularities). The primary ``open string'' object is the Fukaya category $\scrA = \scrF(\pi)$. The two sides are related by open-closed and closed-open string maps
\begin{align} \label{eq:oc0}
& \mathit{OC}: \mathit{HH}_*(\scrA,\scrA) \longrightarrow \mathit{HF}^{*+n}(E,0), \\
\label{eq:co0}
& \mathit{CO}: \mathit{HF}^*(E,1) \longrightarrow \mathit{HH}^*(\scrA,\scrA).
\end{align}
Here, $\mathit{HH}_*(\scrA,\scrA)$ and $\mathit{HH}^*(\scrA,\scrA)$ are the Hochschild homology and cohomology of $\scrA$ (in spite of the notation, the grading of Hochschild homology is cohomological in nature). 

The autoequivalence of $\scrA$ induced by $\sigma$ can be characterized as the Serre functor on that category, up to a shift by $n$ in the grading. In other words, the graph bimodule of $\sigma$ (this has underlying complexes of the form $\mathit{hom}_{\scrA}(\cdot, \sigma \cdot)$, and is an invertible bimodule, with respect to tensor product) is quasi-isomorphic to the shifted dual diagonal bimodule $\scrS = \scrA^\vee[-n]$. One can define Hochschild (co)homology with coefficients in any $\scrA$-bimodule $\scrQ$, denoted by $\mathit{HH}_*(\scrA,\scrQ)$ and $\mathit{HH}^*(\scrA,\scrQ)$. The only case we will need is when $\scrQ = \scrS^r$ is a tensor power of $\scrS$ (or, for $r<0$, of its inverse). Then, one expects to have twisted open-closed and closed-open string maps
\begin{align}
\label{eq:oc-r}
& \mathit{OC}_r: \mathit{HH}_*(\scrA, \scrS^r) \longrightarrow \mathit{HF}^{*+n}(E,-r), \\
\label{eq:co-r}
& \mathit{CO}_r: \mathit{HF}^*(E,1\!-\!r) \longrightarrow \mathit{HH}^*(\scrA, \scrS^r),
\end{align}
which specialize to the previous ones for $r = 0$. There are duality isomorphisms (the first is geometric; the second is algebraic, and holds for all invertible bimodules $\scrQ$)
\begin{align}
& \mathit{HF}^*(E,r) \iso \mathit{HF}^{2n-*}(E,-r)^\vee, \\
\label{eq:algebra-pairing}
& \mathit{HH}^*(\scrA, \scrA^\vee \otimes_{\scrA} \scrQ) \iso 
\mathit{HH}_{-*}(\scrA, \scrQ^{-1})^\vee.
\end{align}
Through these isomorphisms, $\mathit{CO}_r$ should become the dual of $\mathit{OC}_{1-r}$. 
As a related remark, consider the composition of the maps in both directions, which in view of \eqref{eq:algebra-pairing} can be written as
\begin{equation} \label{eq:cooc}
\mathit{CO}_r \circ \mathit{OC}_{r-1}: \mathit{HH}_*(\scrA, \scrS^{r-1}) \longrightarrow \mathit{HH}^{*+n}(\scrA,\scrS^r) \iso \mathit{HH}_*(\scrA, \scrS^{1-r})^\vee.
\end{equation}
A map between the same groups also exists for purely algebraic reasons, as a twisted version of the standard pairing on Hochschild homology. One expects this to agree with \eqref{eq:cooc}.

\subsection{Past work}
Let's review some existing approaches towards defining open-closed and closed-open string maps for Lefschetz fibrations (this discussion will include unpublished results; any errors should be imputed to this author's ignorance). Only exact Lefschetz fibrations, where the smooth fibres are Liouville domains, will be allowed from now on, since most work has been done in that context.

In \cite{seidel00}, $\scrA$ was restricted to a single basis of Lefschetz thimbles (which makes its definition technically easier, and that of $\mathit{OC}$ trivial). In that context, \cite{seidel00b} gives a conjectural geometric formula for $\mathit{HH}^*(\scrA,\scrA)$, corresponding to $\mathit{HF}^*(E,1)$ in our notation. A version of $\mathit{CO}$ for that formalism was constructed by Perutz \cite{perutz10}.

A groundbreaking contribution was made by Abouzaid-Ganatra \cite{abouzaid-ganatra14}. They used a definition of $\scrA$ following Abouzaid-Seidel's earlier unpublished work. This means that morphisms are defined as direct limits over Hamiltonian isotopies which rotate by less than a fixed amount near infinity on $B$. In this framework, they construct $\mathit{OC}_{-1}$ as well as $\mathit{CO}_0$, and analyze the composition \eqref{eq:cooc} by a Cardy relation argument, modelled on that in \cite{abouzaid10}. They use that to derive a split-generation criterion for full subcategories, again along the lines of \cite{abouzaid10}. The pairings on twisted Hochschild homology groups, which we have mentioned above, also arose as part of this program.

Again restricting to a single basis of Lefschetz thimbles, a definition of $\mathit{OC}_r$ was given in \cite{seidel14b} for $r = 1,2$. By duality, this corresponds to $\mathit{CO}_r$ for $r = 0,-1$. The emphasis was on $\mathit{CO}_1$, and how that gives rise to additional structure on $\scrA$ for Lefschetz fibrations that can be compactified by adding a fibre at infinity.


Sylvan introduced another approach to $\scrA$ via ``stops'' \cite{sylvan16}, in which the relevant Floer cohomology groups are obtained by equipping their ``wrapped'' counterparts with a filtration by winding number, and then considering only the subspace where that winding number is $0$. In \cite[Section 4.4]{sylvan16}, a version of $\mathit{OC}$ in that framework is constructed.
%

Most recently, Ganatra, Pardon and Shende \cite{ganatra-pardon-shende17} have considered ``Liouville sectors'' (which, like ``stops'', include exact Lefschetz fibrations as a special case), and defined $\mathit{OC}$ in that context. Their approach to the associated Fukaya categories is again through direct limits. The fundamental contribution of that paper is covariant functoriality both in the closed and open string versions, leading to a local-to-global criterion for $\mathit{OC}$ to be an isomorphism.

\subsection{Contents of this paper}
Again restricting to exact symplectic manifolds (for technical simplicity), we will give a definition of $\scrA$ based on ideas from \cite{seidel17}. This avoids algebraic gadgets such as quotient categories or filtrations, and merely involves a careful choice of almost complex structures and inhomogeneous terms. The underlying geometric viewpoint is that what matters is {\em using a sufficiently large group of automorphisms of the base.} Ideally, one would allow automorphisms whose asymptotic behaviour is governed by any oriented diffeomorphism of the circle at infinity. That presents some technical problems, having to do with preventing pseudo-holomorphic curves from escaping to infinity. Even though those problems are presumably not unsurmountable, we opt for a compromise solution instead, which is sufficient for most purposes (however, see Remark \ref{th:orlov}): namely, to treat the base as a copy of the open complex half-plane (or disc), and to use only hyperbolic isometries of that space at infinity.

There are actually two versions of this framework in the paper: the first one fixes a privileged ``point at infinity'', while the second one is more symmetrical. Section \ref{sec:affine} describes the first version, which leads to a particularly simple definition of the Fukaya category, as explained in Section \ref{sec:fukaya}. The second version is introduced in Section \ref{sec:hyperbolic}, and Section \ref{sec:fukaya2} generalizes the definition of the Fukaya category accordingly. The more general framework naturally accommodates the construction of $\mathit{CO}$, also described in Section \ref{sec:fukaya2}. This expository structure entails a certain amount of repetition (with variations); the advantage is that readers can encounter the basic ideas first in their simplest form (and in particular, those interested only in the Fukaya category itself can focus on Sections \ref{sec:affine} and \ref{sec:fukaya}).

Concerning future developments, we expect our approach to be well-suited for establishing the relation between operations on $\mathit{HF}^*(E,1)$, as constructed in \cite{seidel17}, and their classical counterparts (the Gerstenhaber algebra structure) for $\mathit{HH}^*(\scrA,\scrA)$, via $\mathit{CO}$. The same should apply to the more general maps $\mathit{CO}_r$, which however are not treated in this paper. Eventually, the intended application is to compare the connections on $\mathit{HF}^*(E,r)$ constructed in \cite{seidel17} (for symplectic fibrations with closed fibres and vanishing first Chern class, which is a little different from the setup here) with their categorical counterparts.

{\em Acknowledgments.} This work was partially supported by the Simons Foundation, through a Simons Investigator award; by NSF grant DMS-1500954; and by Columbia University, Princeton University, and the Institute for Advanced Study, through visiting appointments. 

\section{Affine transformations\label{sec:affine}}

In every two-dimensional TQFT type construction, the algebraic structure of the theory depends on what surfaces are allowed, and what additional data they carry. Specifically, to get an $A_\infty$-category structure, one needs to consider discs with marked boundary points; and any structures on those discs, other than the labeling of boundary intervals with objects, need to be topologically inessential (belong to weakly contractible spaces). In this section, we introduce one geometric setup where these conditions holds. Eventually, this will underlie the definition of the Fukaya category of a Lefschetz fibration. Towards that goal we undertake some warmup exercises, involving maps from Riemann surfaces to the upper half plane.

\subsection{Data associated to the ends}
Let 
\begin{equation}
G_{\aff} \iso \bR \rtimes \bR^{>0}
\end{equation}
be the group of orientation-preserving affine transformations of the real line, and $\frakg_{\aff}$ its Lie algebra. If $S$ is a manifold, a one-form $A \in \Omega^1(S,\frakg_\aff)$ can be considered as a connection on the trivial $G_\aff$-bundle over $S$: more precisely, our convention is that the associated connection is $d-A$. The curvature is then given by
\begin{equation} \label{eq:curvature}
F_A = -dA + \half [A,A] \in \Omega^2(S,\frakg_\aff).
\end{equation}
Gauge transformations $\Phi \in \smooth(S,G_\aff)$ act on connections by 
\begin{equation} \label{eq:gauge}
A \longmapsto \Phi_*A = \Phi A \Phi^{-1} + (d\Phi) \Phi^{-1}.
\end{equation}
%
Let's specialize to the case where our manifold is the interval $[0,1]$, and write $A = a_t \mathit{dt}$. By integrating our connection, one obtains a path $\Phi_t \in G$; concretely,
\begin{equation}
\Phi_0 = \Id, \qquad
d\Phi_t/dt = a_t \Phi_t.
\end{equation}
We refer to $g = \Phi_1 \in G_\aff$ as the parallel transport map of $A$ along $[0,1]$. For a connection on a general manifold $S$, a parallel transport is associated to any path $[0,1] \rightarrow S$.

The case of the interval is directly related to the one-dimensional part of our TQFT. Let's consider triples $(A,\lambda_0,\lambda_1)$ consisting of
\begin{equation} \label{eq:triples}
A \in \Omega^1([0,1], \frakg_{\aff}), \quad \lambda_0, \lambda_1 \in \bR,
\end{equation}
subject to one condition. Namely, let $g$ be the parallel transport map of $A$. Then:
\begin{equation} \label{eq:to-the-right}
\parbox[t]{36em}{
the preimage $\lambda_1^\dag = g^{-1}(\lambda_1)$ should lie to the left of $\lambda_0$: $\lambda_0 > \lambda_1^\dag$.}
\end{equation}
Let $\scrP_{\aff}([0,1])$ be the space of all such triples. It carries an action of the gauge group $\scrG_{\aff}([0,1]) = \smooth([0,1],G_{\aff})$, given by $(A,\lambda_0,\lambda_1) \mapsto (\Phi_*A,\Phi_0(\lambda_0), \Phi_1(\lambda_1))$. One sees easily that this action is simply transitive. In particular, $\scrP_{\aff}([0,1])$ is weakly contractible. As a variation, one could also fix $(\lambda_0,\lambda_1)$, and consider the space $\scrA_{\aff}([0,1])$ of all those $A$ such that $(A,\lambda_0,\lambda_1)$ satisfies the conditions above. This is again weakly contractible; one can prove that directly, or use the fact that it fits into a weak fibration
\begin{equation} \label{eq:endpoints-fixed}
\scrA_\aff([0,1]) \longrightarrow \scrP_\aff([0,1]) \xrightarrow{(\lambda_0,\lambda_1)} \bR^2.
\end{equation}

\begin{remark}
In principle, one ought to be careful about the choice of topology on infinite-dimensional spaces such as $\scrP_{\aff}([0,1])$. However, for our purpose it is sufficient to know what one means by a smooth map from a finite-dimensional manifold to one of those spaces, which is clear (a smooth family of connections, and so on). In particular, when we talk about ``weak contractibility'' or ``weak fibration'', that is actually meant as a statement about such families. 
\end{remark}

\begin{remark} \label{th:orlov}
For certain purposes (e.g.\ the definition of the ``Orlov functor'' from the Fukaya category of the fibre to that of the Lefschetz fibration), one may want to consider generalizations where $\lambda_0$ and $\lambda_1$ are finite collections of points on $\bR$ (the analogue of \eqref{eq:to-the-right} would say that any point of $\lambda_1^\dag$ should lie to the left of any point of $\lambda_0$). For that to work, it seems that affine transformations should be replaced by more general diffeomorphisms of $\bR$, something that we will not attempt to carry out here.
\end{remark}

\subsection{Boundary-punctured discs}
Next, we introduce the two-dimensional geometry underlying our TQFT. We consider surfaces $S$ which are discs with $d+1 \geq 2$ boundary punctures. These are of the form $S = \bar{S} \setminus \Sigma$, where $\bar{S}$ is a closed oriented disc, and $\Sigma = \{\zeta_0,\dots,\zeta_d\}$ is a set of $(d+1)$ boundary points, numbered compatibly with their cyclic ordering. We write $\partial_jS$ ($j = 0,\dots,d-1$) for the part of $\partial S$ lying between $\zeta_j$ and $\zeta_{j+1}$, and $\partial_dS$ for the part lying between $\zeta_d$ and $\zeta_0$. In addition, $S$ should come with strip-like ends (one negative end, and $d$ positive ones). These ends are proper oriented embeddings with disjoint images
\begin{equation} \label{eq:ends}
\left\{
\begin{aligned}
& \epsilon_0: \bR^{\leq 0} \times [0,1] \longrightarrow S, \\
& \epsilon_1,\dots,\epsilon_d: \bR^{\geq 0} \times [0,1] \longrightarrow S, \\
& \epsilon_j^{-1}(\partial S) = \{ (s,t) \;:\; t = 0,1\}, \\
& \textstyle \lim_{s \rightarrow \pm\infty} \epsilon_j(s,\cdot) = \zeta_j.
\end{aligned}
\right.
\end{equation}
Fix $(A_j,\lambda_{j,0},\lambda_{j,1}) \in \scrP_{\aff}([0,1])$, for $j = 0,\dots,d$. Given those, we want to equip our surface $S$ with a pair $(A,\lambda)$ consisting of
\begin{equation} \label{eq:a-lambda}
A \in \Omega^1(S,\frakg_{\aff}), \quad \lambda \in \smooth(\partial S, \bR),
\end{equation}
such that:
\begin{align}
\label{eq:a-lambda-1}
& \parbox[t]{36em}{$A$ is flat, meaning that \eqref{eq:curvature} vanishes.}
\\ \label{eq:a-lambda-2}
& \parbox[t]{36em}{
Parallel transport for $A$ along any part of the boundary (thought of as acting on $\bR$ by affine transformations) preserves $\lambda$. In other words, $\lambda$ is covariantly constant.} 
\\
\label{eq:a-lambda-3}
& \parbox[t]{36em}{
$\epsilon_j^*A$ is (the pullback by projection to $[0,1]$ of) $A_j$; and $\lambda_{\epsilon_j(s,0)} = \lambda_{j,0}$, $\lambda_{\epsilon_j(s,1)} = \lambda_{j,1}$.}
\end{align}
Let $\scrP_{\aff}(S,\Sigma)$ be the space of such pairs, where the data $(A_j,\lambda_{j,0}, \lambda_{j,1})$ associated to the ends are kept fixed. This carries an action of the group $\scrG_{\aff}(S,\Sigma)$ of those gauge transformations $\Phi \in \smooth(S,G_{\aff})$ which are trivial on the ends.

\begin{prop}
$\scrP_{\aff}(S,\Sigma)$ is weakly contractible.
\end{prop}

\begin{proof}
Let's first consider the larger space $\scrP_{\aff}(S)$ where the behaviour over the ends can be modelled on any $(d+1)$-tuple of elements in $\scrP_{\aff}([0,1])$. This is weakly homotopy equivalent to $\scrP_{\aff}(S,\Sigma)$, because it sits in a weak fibration 
\begin{equation}
\scrP_{\aff}(S,\Sigma) \longrightarrow \scrP_{\aff}(S) \longrightarrow \scrP_{\aff}([0,1])^{d+1}.
\end{equation}
Let $\scrG_{\aff}(S)$ be the (weakly contractible) group of gauge transformations which, on each end, are independent of the first coordinate $s$. Every flat connection that appears in $\scrP_{\aff}(S)$ can be trivialized by a gauge transformation in $\scrG_{\aff}(S)$, which is unique up to constants. The gauge-transformed boundary condition $\lambda^\dag$ is locally constant, and its values on the boundary components give rise to a point of the (contractible) configuration space
\begin{equation}
\scrC_{\aff}(d+1) = \big\{\lambda_0^\dag > \cdots > \lambda_d^\dag \big\} \subset \bR^{d+1}.
\end{equation}
This means that
\begin{equation} \label{eq:trivialize-a}
\scrP_\aff(S) \iso \scrG_\aff(S) \times_{G_{\aff}} \scrC_\aff(d+1),
\end{equation}
which implies the desired result. One can simplify the argument slightly, by using the subgroup $\scrG_{\aff}(S,\bullet) \subset \scrG_{\aff}(S)$ of based gauge transformations, which means ones that are trivial at some base point $\bullet \in S$. Then, instead of \eqref{eq:trivialize-a}, one has
$\scrP_{\aff}(S) \iso \scrG_{\aff}(S,\bullet) \times \scrC_{\aff}(d+1)$.
\end{proof}

One can also consider the situation where, in addition to the $(A_j,\lambda_{j,0},\lambda_{j,1})$, we already have a fixed $\lambda$. The resulting choices of $A$ form a space $\scrA_{\aff}(S,\Sigma)$, which sits in a weak fibration
\begin{equation} \label{eq:fix-lambda}
\scrA_{\aff}(S,\Sigma) \longrightarrow \scrP_{\aff}(S,\Sigma) \longrightarrow \smooth_c(\partial S,\bR).
\end{equation}

\begin{corollary} \label{th:affine-a-space}
$\scrA_{\aff}(S,\Sigma)$ is weakly contractible. \qed
\end{corollary}

\subsection{A bit of hyperbolic geometry}
As a toy model for Lefschetz fibrations, we will take the target space of the theory to be the upper half-plane. Write
\begin{equation}
\begin{aligned}
& W = \{\mathrm{im}(w) > 0\} \subset \bC, \\
& \bar{W} = \{\mathrm{im}(w) \geq 0\} \cup \{\infty\}, \\
& \partial_\infty W = \partial \bar{W} = \bR \cup \{\infty\}.
\end{aligned}
\end{equation}
The $G_{\mathit{aff}}$-action on the real line extends to $\bar{W}$, fixing $\infty$. On the Lie algebra level, we denote by $\bar{X}_\gamma$ the holomorphic vector field on $\bar{W}$ associated to $\gamma \in \frakg_{\mathit{aff}}$, and by $X_\gamma$ its restriction to $W$. The action on $W$ preserves the hyperbolic area form
\begin{equation}
\omega_W = \frac{d\mathrm{re}(w) \wedge d\mathrm{im}(w)}{\mathrm{im}(w)^2},\end{equation}
as well as its primitive
\begin{equation} \label{eq:infinity-primitive}
\theta_W = \frac{d\mathrm{re}(w)}{\mathrm{im}(w)} = -d^c \log(\mathrm{im}(w)).
\end{equation}
Since $L_{X_\gamma}\theta_W = 0$, the Hamiltonian inducing the vector field $X_\gamma$ can be taken to be
\begin{equation} \label{eq:ha}
H_{\gamma} = \theta_W(X_\gamma).%
\end{equation}
These functions are compatible with Poisson brackets:
\begin{equation}
\begin{aligned}
H_{[\gamma_1,\gamma_2]} & = \theta_W([X_{\gamma_1},X_{\gamma_2}]) = -d\theta_W(X_{\gamma_1},X_{\gamma_2}) + X_{\gamma_1}. \theta_W(X_{\gamma_2}) - 
X_{\gamma_2}.\theta_W(X_{\gamma_1}) \\ & = \omega_W(X_{\gamma_1},X_{\gamma_2}) =
\{H_{\gamma_1}, H_{\gamma_2}\}.
\end{aligned}
\end{equation}

\subsection{Geometric structures associated to flat connections\label{subsec:package}}
Let $S$ be an arbitrary connected Riemann surface with boundary; we denote its complex structure by $j$. Let's equip $S$ with a pair $(A,\lambda)$ as in \eqref{eq:a-lambda}, which satisfies \eqref{eq:a-lambda-1}, \eqref{eq:a-lambda-2}. 
Through the Lie algebra homomorphisms $\gamma \mapsto X_\gamma$ and $\gamma \mapsto H_\gamma$, $A$ induces one-forms $X_A$ and $H_A$ on $S$ with values in, respectively, $\smooth(W,TW)$ and $\smooth(W,\bR)$. 
One can think of $X_A$ as an Ehresmann connection on the (trivial) fibre bundle 
\begin{equation} \label{eq:trivial-fibre-bundle}
S \times W \longrightarrow S,
\end{equation}
which lifts any vector field $\xi$ on $S$ to the vector field $\xi + X_A(\xi)$ on $S \times W$. On the Hamiltonian level, flatness of the connection is expressed by the identity
\begin{equation} \label{eq:poisson-flatness}
\partial_t H_A(\partial_s) - \partial_s H_A(\partial_t) + \{ H_A(\partial_s), H_A(\partial_t) \} = 0.
\end{equation}
There is an associated closed $\omega_A \in \Omega^2(S \times W)$, which agrees with $\omega_W$ on each fibre, and which vanishes after contraction with any $\xi + X_A(\xi)$. In local coordinates as before, 
\begin{equation} \label{eq:omega-a}
\begin{aligned}
\omega_A & = \omega_W + \omega_W(X_A(\partial_s), \cdot) \wedge \mathit{ds}
+ \omega_W(X_A(\partial_t),\cdot) \wedge \mathit{dt} -
\omega_W(X_A(\partial_s), X_A(\partial_t)) \mathit{ds} \wedge \mathit{dt} 
\\
& = \omega_W - d\big(H_A(\partial_s) \mathit{ds} + H_A(\partial_t) \mathit{dt}\big).
\end{aligned}
\end{equation}
In the second line of \eqref{eq:omega-a}, we consider $H_A(\partial_s) \mathit{ds} + H_A(\partial_t) \mathit{dt}$ as a one-form on $S \times W$, vanishing in $TW$ direction, and take its exterior derivative. The equality between the two lines uses \eqref{eq:poisson-flatness}. There is an evident choice of primitive,
\begin{equation} \label{eq:a-primitive}
\theta_A = \theta_W - H_A(\partial_s) \mathit{ds} - H_A(\partial_t) \mathit{dt}.
\end{equation}
Similarly, we get a complex structure $J_A$ on $S \times W$, which is such that projection to $S$ is $J_A$-holomorphic, and which restricts to the standard complex structure on each $W$ fibre. It is characterized by those properties, together with the fact that
\begin{equation}
J_A\big(\partial_s + X_A(\partial_s)\big) = \partial_t + X_A(\partial_t).
\end{equation}
The associated pairing $\omega_A(\cdot,J_A\cdot)$ is symmetric and satisfies
\begin{equation}
\omega_A(\sigma, J_A \sigma) \geq 0, 
\end{equation}
with equality iff $\sigma = \xi + X_A(\xi)$ for some $\xi \in TS$. Note that our connection extends to $S \times \bar{W}$, and so does $J_A$; we denote the extensions by $\bar{X}_A$ and $\bar{J}_A$. The function $\lambda$ gives rise to a $J_A$-totally real submanifold
\begin{equation}
\Lambda = \{ (z,w) \in \partial S \times W \; :\; \mathrm{re}(w) = \lambda_z \}.
\end{equation}
Both $\omega_A$ and $\theta_A$ vanish when restricted to $\Lambda$. The closure $\bar\Lambda \subset \partial S \times \bar{W}$, obtained by adding two points at infinity to each fibre, is a submanifold with boundary.

In a way, this discussion has been overkill. If we have a gauge transformation $\Phi \in \smooth(S,G_{\aff})$ and two sets of data 
\begin{equation} \label{eq:dagger}
(A,\lambda) = \Phi_*(A^\dag,\lambda^\dag), 
\end{equation}
then the induced fibrewise automorphism of \eqref{eq:trivial-fibre-bundle} maps all the geometric structures associated to $(A^\dag,\lambda^\dag)$ to their counterparts for $(A,\lambda)$. Locally, one can gauge transform any $(A,\lambda)$ to the trivial choice $(A^\dag,\lambda^\dag) = (0,0)$. In that case, $\omega_{A^\dag}$ and $\theta_{A^\dag}$ are the pullbacks of $\omega_W$ and $\theta_W$ by projection; $J_{A^\dag}$ is the product complex structure; and $\Lambda^\dag = \partial S \times i\bR^{>0}$. This provides easy proofs of several general properties stated above (integrability of $J_A$, and vanishing of $\omega_A|\Lambda$, $\theta_A|\Lambda$).

\subsection{Maps to the half-plane\label{subsec:maps}}
Let $S$ and $(A,\lambda)$ be as before. Consider the following Cauchy-Riemann equation for maps $u: S \rightarrow W$:
\begin{equation} \label{eq:u-equation}
\left\{
\begin{aligned}
& (Du - X_A)^{0,1} = 0, \\
& \mathrm{re}(u(z)) = \lambda_z \quad \text{for $z \in \partial S$.}
\end{aligned}
\right.
\end{equation}
The first line means that for each $z \in S$, $Du_z - X_{A,z}: TS_z \rightarrow TW_{u(z)} = \bC$ is complex-linear. An equivalent way of formulating \eqref{eq:u-equation} is to consider the section $v(z) = (z,u(z))$ of \eqref{eq:trivial-fibre-bundle}. This will be a holomorphic section with totally real boundary conditions:
\begin{equation} \label{eq:v-equation}
\left\{
\begin{aligned}
& J_A \circ Dv = Dv \circ j, \\
& v(\partial S) \subset \Lambda.
\end{aligned}
\right.
\end{equation}
The energy of a solution is
\begin{equation} \label{eq:toy-energy}
E(u) = \int_S \|Du-X_A\|^2_W 
=  \int_S v^*\omega_A.
\end{equation}
Here $\|\cdot\|_W$ is the norm derived from the hyperbolic metric: $\|Du-X_A\|_W = |Du-X_A|/\mathrm{im}(u)$ (our convention for the norm of a linear map $TS \rightarrow \bC$ is half of the usual one, which is why \eqref{eq:toy-energy} is missing the standard $\half$ factor). Because $\theta_A|\Lambda = 0$, $\int_S v^*\omega_A$ is a ``topological'' quantity for sections with boundary values in $\Lambda$; by that, we mean that it is invariant under compactly supported deformations within that space of sections. 

\begin{example}
If $S$ is compact, we have $E(u) = 0$ by Stokes; in that case, the only possible solutions are those which satisfy $Du = X_A$ everywhere.
\end{example}

In the situation of \eqref{eq:dagger}, if $u$ is a solution of \eqref{eq:u-equation} for $(A,\lambda)$, applying $\Phi^{-1}$ pointwise yields a solution $u^\dag$ of the corresponding equation for $(A^\dag,\lambda^\dag)$. Hence, all local considerations can be addressed by reducing to $(A^\dag,\lambda^\dag) = (0,0)$, in which case $u^\dag$ is just a holomorphic function with real boundary condition:
\begin{equation} \label{eq:holo}
\left\{
\begin{aligned}
& \bar\partial u^\dag = 0, \\
& \mathrm{im}(u^\dag) > 0, \\
& \mathrm{re}(u^\dag|\partial S) = 0.
\end{aligned}
\right.
\end{equation}

\begin{lemma} \label{th:schwarz}
Let $u$ be a solution of \eqref{eq:u-equation} defined on the unit disc $S = B = \{|z|<1\} \subset \bC$. Then we have the pointwise bound $\|Du-X_A\|_W \leq 2/(1-|z|^2)$. The same holds for a solution defined on a half-disc $S = C = \{|z| < 1, \; \mathrm{re}(z) \geq 0\}$.
\end{lemma}

\begin{proof}
The first part is obtained by reducing to \eqref{eq:holo}, where it's the classical Schwarz Lemma. For the second part, one additionally uses the reflection principle to extend the solution of \eqref{eq:holo} from $C$ to $B$.
\end{proof}
 
\begin{lemma} \label{th:compactness-0}
Let $u_k: S \rightarrow W$ be a sequence of solutions of \eqref{eq:u-equation}. Suppose that there are points $z_k$ contained in a compact subset of $S$, such that $u_k(z_k) \rightarrow \partial_\infty W$. Then $u_k \rightarrow \partial_\infty W$ uniformly on compact subsets. Moreover, a subsequence converges (in the same sense) to a map $u_\infty: S \rightarrow \partial_\infty W$ which satisfies
\begin{equation}
\left\{
\begin{aligned}
& Du_\infty = \bar{X}_A, \\
& u_\infty(z) = \lambda_z \text{ or } \infty \quad \text{for $z \in \partial S$.}
\end{aligned}
\right.
\end{equation}
\end{lemma}

\begin{proof}
Restrict to a fixed compact subset of $S$, and use any Riemannian metric on $W$ that extends to $\bar{W}$. Then, Lemma \ref{th:schwarz} yields bounds on $Du_k$ with respect to that metric. It follows that a subsequence converges, uniformly on compact subsets, to some $u_\infty: S \rightarrow \bar{W}$, which is again a solution of \eqref{eq:u-equation} (with boundary conditions that are the closure of the previous ones). Moreover, there is at least one point $z_\infty \in S$ for which $u_\infty(z_\infty) \in \partial \bar{W}$. 

Let's apply a gauge transformation locally near $z_\infty$, which relates $(A,\lambda)$ to $(A^\dag,\lambda^\dag) = (0,0)$, and correspondingly $u_\infty$ to a holomorphic map $u^\dag_\infty$. In local coordinates in which $z_\infty = 0$, the situation is one of the following (the notation $B$ and $C$ is taken from Lemma \ref{th:schwarz}):
\begin{align}
&  \label{eq:lim1}
u^\dag_\infty: B \rightarrow \bar{W}, \;\; u^\dag_\infty(0) \in \partial \bar{W}.
\\ & \label{eq:lim2}
u^\dag_\infty: C \rightarrow \bar{W},\;\; u^\dag_\infty(z) \in i\bR^{\geq 0} \cup \{\infty\} \text{ for }
z \in \partial C, \text{ and } u^\dag_\infty(0) = 0. 
\\ & \label{eq:lim3}
\text{As in \eqref{eq:lim2}, but with } u^\dag_\infty(0) = \infty \text{ instead.}
\end{align}
In the first instance, the open mapping principle shows that $u^\dag_\infty$ must be constant. One arrives at the same conclusion for \eqref{eq:lim2}, as follows: after applying the reflection principle to extend $u^\dag_\infty$ to $B$, write $u^\dag_\infty(z) = \sum_{k \geq 1} ia_k z^k$ near $z = 0$, with (because of the boundary condition) $a_k \in \bR$. Suppose that $u^\dag_\infty$ is not constant, and let $a_k$ be the first nonzero coefficient. If $a_k > 0$, then $\mathrm{im}(u^\dag_\infty(r e^{\pi i/k}))<0$ for small $r>0$; and if $a_k < 0$, the same holds for $\mathrm{im}(u^\dag_\infty(r))$. This leads to a contradiction. Finally, \eqref{eq:lim3} can be reduced to \eqref{eq:lim2} by passing to $-1/u^\dag_\infty$. 

Translating back, we see that near $z_\infty$, $u_\infty$ takes values in $\partial_\infty W$ and satisfies $Du_\infty = \bar{X}_A$. By the same argument, the subset where this holds is open and closed, hence everything. We have now shown that a subsequence converges to $\partial_\infty W$; but since that applies to any choice of subsequence of the original sequence as well, it follows that $u_k \rightarrow \partial_\infty W$.
\end{proof}

\section{Hyperbolic isometries\label{sec:hyperbolic}}

The affine automorphisms of the upper half plane, on which our previous construction was based, can be thought of as hyperbolic isometries fixing a point at infinity. Picking such a privileged point breaks the natural symmetry, and that eventually becomes an obstruction for further developments. We will therefore revisit our setup, dropping that restriction.

\subsection{Rational transformations and their lifts}
We will use the group 
\begin{equation}
G = \mathit{PSL}_2(\bR)
\end{equation}
of (orientation-preserving) rational transformations of $\bR P^1 = \bR \cup \{\infty\}$, and its Lie algebra $\frakg$. For convenience, let's identify $\bR P^1 = \bR/2\pi\bZ$ (so that $\infty$ corresponds to $\pi$). On the universal cover $\tilde{G} \rightarrow G$, the action lifts to $\bR \rightarrow \bR P^1$. The role of \eqref{eq:triples} in our new context will be played by triples $(A,\tilde\lambda_0,\tilde\lambda_1)$ consisting of
\begin{equation}
A \in \Omega^1([0,1], \frakg), \quad \tilde\lambda_0, \tilde\lambda_1 \in \bR,
\end{equation}
with the following property. Let $\tilde{g} \in \tilde{G}$ be the natural lift of the parallel transport map $g \in G$ of $A$ (this exists since $g$ is the endpoint of a path starting at the identity). Write $\tilde\lambda_1^\dag = \tilde{g}^{-1}(\tilde\lambda_1)$. We require that:
\begin{equation}
\parbox[t]{36em}{
$\tilde{\lambda}_0$ and $\tilde\lambda_1^\dag$ should map to distinct points in $\bR P^1$. Moreover,
$\tilde{\lambda}_0$ should be the first point in its fibre over $\bR P^1$ that can be reached from $\tilde{\lambda}_1^\dag$ by moving in positive direction: this means that $\tilde{\lambda}_0 - \tilde{\lambda}_1^\dag \in (0,2\pi)$.}
\end{equation}
Let $\scrP([0,1])$ be the space of all such triples. It carries an action of $\scrG([0,1]) = \smooth([0,1],\tilde{G})$. That action is easily seen to be transitive, with each point having a stabilizer isomorphic to $\bR$ (the subgroup of $G$ that fixes two distinct points on $\bR P^1$). As a consequence, $\scrP([0,1])$ is weakly contractible. One can also consider the subspace $\scrA([0,1])$ where $(\tilde{\lambda}_0,\tilde{\lambda}_1)$ is fixed, which is again weakly contractible, by an argument parallel to \eqref{eq:endpoints-fixed}.

\begin{remark}
As mentioned before, one can identify $G_\aff$ with the subgroup of $G$ that fixes $\infty$. Elements of $G_\aff$ have preferred preimages in $\tilde{G}$. Similarly, any $\lambda \in \bR P^1 \setminus \{\infty\}$ has a unique lift $\tilde{\lambda} \in (-\pi,\pi) \subset \bR$. Hence, $\scrP_\aff([0,1])$ can be considered as a subset of $\scrP([0,1])$.
%
\end{remark}

Any $\tilde{g} \in \tilde{G}$ has an associated rotation number $\mathrm{rot}(\tilde{g}) \in \bR$. We will be only interested in the situation where the underlying $g \in G$ is hyperbolic, in which case the rotation number is an integer. More precisely, let $l_{\mathit{small}},l_{\mathit{big}} \in \bR P^1$ be the two eigenvectors of $g$, where the convention is that $l_{\mathit{small}}$ belongs to the eigenvalue with absolute value $<1$. Then, for any $\tilde{l} \in \bR$ with image $l \in \bR P^1$,
\begin{equation} \label{eq:rotation-number-and-shift}
\tilde{g}(\tilde{l}) - \tilde{l} - 2\pi \mathrm{rot}(\tilde{g})
\begin{cases}
= 0 & \text{if $l = l_{\mathit{small}}$ or $l_{\mathit{big}}$,} \\
\in (0,2\pi) & \text{if $l \in (l_\mathit{small},l_\mathit{big})$,} \\
\in (-2\pi,0) & \text{if $l \in (l_\mathit{big},l_\mathit{small})$.}
\end{cases}
\end{equation}
Here, $(l_\mathit{small},l_\mathit{big}) \subset \bR/2\pi\bZ$ stands for the open interval in the circle bounded by $l_\mathit{small}$ on the left and $l_\mathit{big}$ on the right; and correspondingly for $(l_\mathit{big},l_\mathit{small})$. With these preliminaries at hand, we can introduce the ``closed string'' analogue of the previous definition. For $\tau>2$, let $\scrP_\tau(S^1)$ be the space of those $A \in \Omega^1(S^1,\frakg)$ such that:
\begin{equation}
\parbox[t]{36em}{The holonomy (parallel transport around $S^1$) of $A$ is a hyperbolic element $g \in G$, with $|\mathrm{tr}(g)| = \tau$; and its natural lift $\tilde{g}$ has rotation number $1$.}
\end{equation}
The group $\scrG(S^1) = \smooth(S^1,\tilde{G})$ acts transitively on this space, and each point has stabilizer isomorphic to $\bZ \times \bR$. Hence, we get a weak homotopy equivalence
\begin{equation}
\scrP_\tau(S^1) \htp \bR P^1,
\end{equation}
which can be realized by mapping each $A$ to an eigenvector (either $l_\mathit{small}$ or $l_\mathit{big}$) of $g$.

\subsection{Boundary-punctured discs revisited}
Let $S$ be a disc with $(d+1)$ boundary punctures, and strip-like ends \eqref{eq:ends}. Fix $(A_j, \tilde{\lambda}_{j,0}, \tilde{\lambda}_{j,1}) \in \scrP([0,1])$ for $j = 0,\dots,d$. Given that, we consider pairs $(A,\tilde{\lambda})$, where
\begin{equation} \label{eq:tilde-pair}
A \in \Omega^1(S,\frakg), \quad \tilde\lambda \in \smooth(\partial S, \bR).
\end{equation}
Denote the image of $\tilde\lambda$ by $\lambda \in \smooth(\partial S, \bR P^1)$. We impose the following analogues of \eqref{eq:a-lambda-1}--\eqref{eq:a-lambda-3}:
\begin{align}
& \label{eq:new-a-1}
\parbox{36em}{$A$ is flat.}
\\ & \label{eq:new-a-2}
\parbox{36em}{Parallel transport along any part of $\partial S$ preserves $\lambda$.}
\\ & \label{eq:new-a-3}
\parbox{36em}{$\epsilon^*_jA = A_j$; and $\tilde{\lambda}_{\epsilon_j(s,0)} = \tilde\lambda_{j,0}$, $\tilde{\lambda}_{\epsilon_j(s,1)} = \tilde\lambda_{j,1}$.}
\end{align}
Let $\scrP(S,\Sigma)$ be the space of all such \eqref{eq:tilde-pair}. It carries an action of the group $\scrG(S,\Sigma)$ of those $\tilde\Phi \in \smooth(S,\tilde{G})$ such that, for each $j$: $\tilde\Phi_{\epsilon_j(s,t)}$ is independent of $s$, and lies in the subgroup of $\scrG([0,1])$ which stabilizes $(A_j, \tilde{\lambda}_{j,0}, \tilde{\lambda}_{j,1})$. 

\begin{prop}
$\scrP(S,\Sigma)$ is weakly contractible.
\end{prop}

\begin{proof}
Let $\scrP(S)$ be the larger space where the behaviour over the ends can be modelled on any $(d+1)$-tuple of elements in $\scrP([0,1])$. This is weakly homotopy equivalent to $\scrP(S,\Sigma)$, because it sits in a weak fibration
\begin{equation}
\scrP(S,\Sigma) \longrightarrow \scrP(S) \longrightarrow \scrP([0,1])^{d+1}.
\end{equation}
Let $\scrG(S,\bullet)$ be the group of those $\tilde{\Phi} \in \smooth(S,\tilde{G})$ which, on each end, are independent of $s$, and which are trivial at some base point. Using such gauge transformations to trivialize the connection, we get that
$\scrP(S) \iso \scrG(S,\bullet) \times \scrC(d+1)$,
where 
\begin{equation} \label{eq:tilde-c-2}
\scrC(d+1) = \big\{\tilde{\lambda}_0^\dag > \cdots > \tilde{\lambda}_d^\dag, \;\;
\tilde{\lambda}_0^\dag - \tilde{\lambda}_d^\dag \in (0,2\pi) \big\} \subset \bR^{d+1}.
\end{equation}
Both $\scrG(S,\bullet)$ and $\scrC(d+1)$ are weakly contractible, and this implies the desired result.
\end{proof}

Suppose that, in addition to the $(A_j,\tilde{\lambda}_{j,0},\tilde{\lambda}_{j,1})$, we already have a fixed $\tilde{\lambda}$. The remaining space $\scrA(S,\Sigma)$ of all possible choices of $A$ satisfies the analogue of \eqref{eq:fix-lambda}, hence:

\begin{corollary} \label{th:locally-constant-1}
$\scrA(S,\Sigma)$ is weakly contractible. \qed
\end{corollary}

\subsection{Adding an interior puncture}
A disc with $d+1 \geq 1$ boundary punctures and one interior puncture is a surface of the form $S = \bar{S} \setminus \Sigma$, where $\bar{S}$ is (again) a closed disc, and $\Sigma = \{\zeta_0,\dots,\zeta_{d+1}\}$ consists of boundary points $\zeta_0,\dots,\zeta_d$ (numbered as before), together with an interior point $\zeta_{d+1}$. A set of ends for such a surface consists of \eqref{eq:ends} and an additional embedding (with image disjoint from the others)
\begin{equation} \label{eq:ends2}
\left\{
\begin{aligned}
& \epsilon_{d+1}: \bR^{\geq 0} \times S^1 \longrightarrow S, \\
& \textstyle \lim_{s \rightarrow \infty} \epsilon_{d+1}(s,\cdot) = \zeta_{d+1}.
\end{aligned}
\right.
\end{equation}
Fix $(A_j,\tilde\lambda_{j,0},\tilde\lambda_{j,1}) \in \scrP([0,1])$, for $j = 0,\dots,d$, as well as $A_{d+1} \in \scrP_\tau(S^1)$, for some $\tau>2$. We consider pairs $(A,\tilde\lambda)$ on $S$ as in \eqref{eq:tilde-pair}, but where the condition $\epsilon_j^*A = A_j$ is also applied to \eqref{eq:ends2}. We again write $\scrP(S,\Sigma)$ for the space of such pairs (the notation is as before, but we are looking at a different kind of surface).
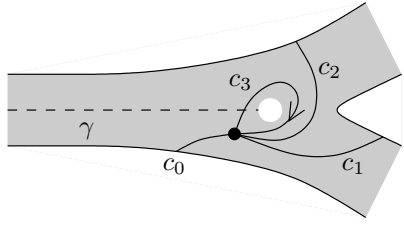
\begin{figure}
\begin{centering}
\begin{picture}(0,0)%
\includegraphics{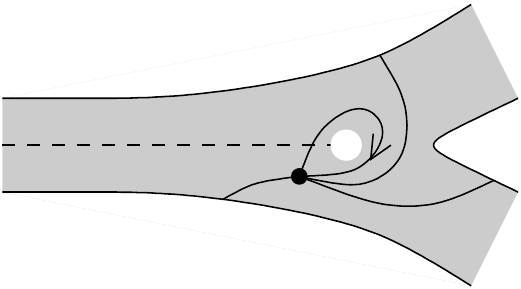}%
\end{picture}%
\setlength{\unitlength}{3947sp}%
\begingroup\makeatletter\ifx\SetFigFont\undefined%
\gdef\SetFigFont#1#2#3#4#5{%
  \reset@font\fontsize{#1}{#2pt}%
  \fontfamily{#3}\fontseries{#4}\fontshape{#5}%
  \selectfont}%
\fi\endgroup%
\begin{picture}(2499,1374)(1339,-1273)
\put(1801,-736){\makebox(0,0)[lb]{\smash{{\SetFigFont{10}{12.0}{\rmdefault}{\mddefault}{\updefault}{\color[rgb]{0,0,0}$\gamma$}%
}}}}
\put(2326,-961){\makebox(0,0)[lb]{\smash{{\SetFigFont{10}{12.0}{\rmdefault}{\mddefault}{\updefault}{\color[rgb]{0,0,0}$c_0$}%
}}}}
\put(3451,-961){\makebox(0,0)[lb]{\smash{{\SetFigFont{10}{12.0}{\rmdefault}{\mddefault}{\updefault}{\color[rgb]{0,0,0}$c_1$}%
}}}}
\put(3301,-361){\makebox(0,0)[lb]{\smash{{\SetFigFont{10}{12.0}{\rmdefault}{\mddefault}{\updefault}{\color[rgb]{0,0,0}$c_2$}%
}}}}
\put(2741,-436){\makebox(0,0)[lb]{\smash{{\SetFigFont{10}{12.0}{\rmdefault}{\mddefault}{\updefault}{\color[rgb]{0,0,0}$c_3$}%
}}}}
\end{picture}%
\caption{\label{fig:cut}The paths from the proof of Proposition \ref{th:puncture}.}
\end{centering}
\end{figure}

\begin{prop} \label{th:puncture}
$\scrP(S,\Sigma)$ is weakly homotopy equivalent to $\bZ$.
\end{prop}

\begin{proof}
Let $\scrP_\tau(S)$ be the larger space where the data prescribing the behaviour on the ends may vary, but still keeping $\tau$ fixed. By definition, this sits in a weak fibration
\begin{equation} \label{eq:mixed-fibration}
\scrP(S,\Sigma) \longrightarrow \scrP_\tau(S) \longrightarrow \scrP([0,1])^{d+1} \times \scrP_\tau(S^1).
\end{equation}
Choose a base point $\bullet \in S$. Fix a curve $\gamma$ in the interior of $S$, connecting the ends $\zeta_0$ and $\zeta_{d+1}$ (cutting open the surface along that curve would make it contractible), and avoiding $\bullet$. Choose paths $c_0,\dots,c_d$ from each of the boundary components to $\bullet$, all disjoint from $\gamma$ (this determines their homotopy classes). We also fix a loop $c_{d+1}$ based at $\bullet$, which goes clockwise once around $\zeta_{d+1}$; see Figure \ref{fig:cut}.
Given $(A,\tilde\lambda) \in \scrP_\tau(S)$, move $\tilde{\lambda}_{c_j(0)}$, $j = 0,\dots,d$, by parallel transport along $c_j$ to $\bullet$, and denote the outcome by $\tilde\lambda_j^\dag \in \bR$. Additionally, let $\tilde{g}^\dag \in \tilde{G}$ be the holonomy around $c_{d+1}$, and $g^\dag$ its image in $G$. These satisfy:
\begin{equation} \label{eq:tilde-holonomy}
\left\{
\begin{aligned}
& \tilde{g}^\dag \text{ is hyperbolic with rotation number $1$, and $|\mathrm{tr}(g^\dag)| = \tau$;} \\
& \tilde{\lambda}_0^\dag - (\tilde{g}^\dag)^{-1}(\tilde{\lambda}_d^\dag) \in (0,2\pi); \\
& \tilde{\lambda}_j^\dag - \tilde{\lambda}_{j+1}^\dag \in (0,2\pi) \quad \text{for $j = 0,\dots,d-1$}.
\end{aligned}
\right.
\end{equation}
Let $\scrC_\tau(d+1;1)$ be the space of all solutions of \eqref{eq:tilde-holonomy}. Let $\scrG(S,\bullet)$ be the group of those $\tilde\Phi \in \smooth(S,\tilde{G})$ which are trivial at $\bullet$, and independent of $s$ on each end. This acts freely on $\scrP_\tau(S)$, and the process defined above yields a $\scrG(S,\bullet)$-invariant map $\scrP_\tau(S) \rightarrow \scrC_\tau(d+1;1)$, which is easily seen to be onto. Now, take any two points of $\scrP_\tau(S)$ lying in the same fibre of that map. Because the monodromies are the same, the two connections can be related by a (unique) gauge transformation which lies in $\scrG(S,\bullet)$. It then follows that the gauge transformation also relates the boundary conditions $\tilde\lambda$. The consequence is that we have a weak fibration
\begin{equation} \label{eq:based-tau}
\scrG(S,\bullet) \longrightarrow \scrP_\tau(S) \longrightarrow \scrC_\tau(d+1;1).
\end{equation}

Our main task is to analyze $\scrC_\tau(d+1;1)$. Since $(\tilde{g}^\dag)^{-1}$ has rotation number $-1$, it moves every point on the real line to the left, compare \eqref{eq:rotation-number-and-shift}; which yields the implication
\begin{equation}
\tilde{\lambda}_0^\dag - (\tilde{g}^\dag)^{-1}(\tilde{\lambda}_d^\dag) < 2\pi \;\; \Longrightarrow \;\;
\tilde{\lambda}_0^\dag - \tilde{\lambda}_d^\dag < 2\pi. 
\end{equation}
Hence, only part of the last line in \eqref{eq:tilde-holonomy} is necessary, namely that $\tilde{\lambda}_j^\dag > \tilde{\lambda}_{j+1}^\dag$. We also have
\begin{equation} \label{eq:big-small}
\tilde{\lambda}_0^\dag - (\tilde{g}^\dag)^{-1}(\tilde{\lambda}_d^\dag) < 2\pi \;\; \text{ and (for $d>0$) } \;\;
\tilde{\lambda}_0^\dag > \tilde{\lambda}_d^\dag \;\; \Longrightarrow \;\;
\tilde{\lambda}_0^\dag - (\tilde{g}^\dag)^{-1}(\tilde{\lambda}_0^\dag) < 2\pi. 
\end{equation}
In terms of \eqref{eq:rotation-number-and-shift}, this says that the image of $\tilde\lambda_0^\dag$ in $\bR P^1$ belongs to the interval $(l_{\mathit{big}}, l_{\mathit{small}})$ bounded by the eigenvectors of $\tilde{g}^\dag$. Taking this into account, one can rewrite \eqref{eq:tilde-holonomy} as:
\begin{equation} \label{eq:tilde-holonomy-2}
\left\{
\begin{aligned}
& \tilde{g}^\dag \text{ is hyperbolic, with rotation number $1$;} \\
& \tilde{\lambda}_0^\dag \text{ is a preimage of a point in $(l_{\mathit{big}}, l_{\mathit{small}})$;} \\
& \tilde{\lambda}_0^\dag - (\tilde{g}^\dag)^{-1}(\tilde{\lambda}_d^\dag) \in (0,2\pi); \\
& \tilde{\lambda}_0^\dag > \tilde{\lambda}_1^\dag > \cdots > \tilde{\lambda}_d^\dag.
\end{aligned}
\right.
\end{equation}
From this point of view, the construction of a point in $\scrC_\tau(d+1;1)$ proceeds in the following steps:
\begin{align}
&
\parbox{36em}{Choose $\tilde{\lambda}_0^\dag \in \bR$.} 
\\ &
\parbox{36em}{Next, take $l_{\mathit{small}} \neq l_{\mathit{big}}$ in $\bR P^1$, so that $\tilde{\lambda}_0^\dag$ lies in the preimage of $(l_{\mathit{big}}, l_{\mathit{small}})$.}
\\ &
\parbox[t]{36em}{Take the unique $\tilde{g}^\dag$ in the specific hyperbolic conjugacy class, whose eigenvectors are the given $l_{\mathit{small}}, \, l_{\mathit{big}}$. This will automatically satisfy $\tilde\lambda_0^\dag - (\tilde{g}^\dag)^{-1}(\tilde{\lambda}_0^\dag) \in (0,2\pi)$.}
\\ &
\parbox[t]{36em}{If $d>0$, fix $\tilde{\lambda}_d^\dag < \tilde{\lambda}_0^\dag$ satisfying the third line of \eqref{eq:tilde-holonomy-2}. This is always possible, since any $\tilde{\lambda}_d^\dag$ which is sufficiently close to $\tilde{\lambda}_0^\dag$ will have that property.} 
\\ &
\parbox[t]{36em}{If $d>1$, choose $\tilde{\lambda}_1^\dag > \cdots > \tilde{\lambda}_{d-1}^\dag$ in the interval $(\tilde{\lambda}_0^\dag, \tilde{\lambda}_d^\dag)$.}
\end{align}
Since all choices belong to contractible spaces, $\scrC_\tau(d+1;1)$ is contractible. This, together with the weak contractibility of $\scrG(S,\bullet)$, and the fact that the base of \eqref{eq:mixed-fibration} is weakly homotopy equivalent to a circle, implies the desired result.
\end{proof}

\begin{addendum} \label{th:explicit-punctured}
In the proof of Proposition \ref{th:puncture}, let's take $\bullet = \epsilon_{d+1}(s,0)$ for some $s$, and use the loop $c_{d+1}(t) = \epsilon_{d+1}(s,t)$. If we start with $(A,\tilde\lambda) \in \scrP(S,\Sigma)$, the resulting $\tilde{g}$ will always be the holonomy of $A_{d+1}$, so the eigenvectors $l_{\mathit{small}}$ and $l_{\mathit{big}}$ are fixed throughout $\scrP(S,\Sigma)$. Fix an identification between the set of connected components of the preimage of $(l_{\mathit{big}}, l_{\mathit{small}})$ and $\bZ$, compatible with the covering action. By \eqref{eq:big-small}, $\tilde{\lambda}_0^\dag$ lies in such a connected component, hence giving rise to a map
\begin{equation} \label{eq:pick-sheet}
\scrP(S,\Sigma) \longrightarrow \bZ.
\end{equation}
Inspection of the proof of Proposition \ref{th:puncture} shows that this is a weak homotopy equivalence (and independent of all choices up to a constant). In a nutshell, the argument is as follows: consider the space $\tilde{\scrI}$ of open intervals of length $<2\pi$ in $\bR$, and the corresponding space $\scrI$ of intervals in $\bR P^1$. Then, \eqref{eq:pick-sheet} sits in a commutative diagram
\begin{equation}
\xymatrix{
\ar[d] \scrP(S,\Sigma) \ar[r] & \ar[r] \ar[d] \scrP_\tau(S) & \scrP([0,1])^{d+1} \times \scrP_\tau(S^1) \ar[d] \\
\bZ \ar[r] & \tilde{\scrI} \ar[r] & \scrI
}
\end{equation}
The rows are weak fibrations, and the middle and right vertical arrows are weak homotopy equivalences; hence so is that on the left.
\end{addendum}

One can use Addendum \ref{th:explicit-punctured} to derive the following consequence. Let's introduce a parameter $\theta \in S^1$ which rotates the end $\epsilon_{d+1}$. Correspondingly, one has a parametrized space $\scrP_{\mathit{rotate}}(S,\Sigma)$, which sits in a weak fibration
\begin{equation} \label{eq:rotate-fibration}
\scrP(S,\Sigma) \longrightarrow \scrP_{\mathit{rotate}}(S,\Sigma) \longrightarrow S^1.
\end{equation}

\begin{cor}
Under the weak homotopy equivalence \eqref{eq:pick-sheet}, the fibration \eqref{eq:rotate-fibration} has holonomy around $S^1$ which shifts the sheets by $1$. Hence, $\scrP_{\mathit{rotate}}(S,\Sigma)$ is weakly contractible. \qed
\end{cor}

As usual, one can also consider the spaces $\scrA(S,\Sigma)$ and $\scrA_{\mathit{rotate}}(S,\Sigma)$ where $\tilde{\lambda}$ is kept fixed, and get corresponding results:

\begin{corollary} \label{th:last-contractible}
$\scrA(S,\Sigma)$ is weakly homotopy equivalent to $\bZ$, with an explicit homotopy equivalence given by \eqref{eq:pick-sheet}. Moreover, $\scrA_{\mathit{rotate}}(S,\Sigma)$ is weakly contractible. \qed
\end{corollary}

\subsection{A bit more hyperbolic geometry}
From this point onwards, we find it convenient to switch to the disc model for the hyperbolic plane:
\begin{equation}
\begin{aligned}
& B = \{|w| < 1\} \subset \bC, \\
& \bar{B} = \{|w| \leq 1\}, \\
& \partial_\infty B = \partial \bar{B},
\end{aligned}
\end{equation}
and to replace $\mathit{PSL}_2(\bR)$ by the isomorphic group (using the same notation, which hopefully does not cause too much confusion)
\begin{align}
& G = \mathit{PU}(1,1) = \left\{ \left(\begin{smallmatrix} a & b \\ \bar{b} & \bar{a} \end{smallmatrix}\right)\; : \;
a,b \in \bC, \; |a|^2 - |b|^2 = 1 \right\}/\pm\!\Id, \\
&
\frakg = \mathfrak{su}(1,1) = \left\{ \left(\begin{smallmatrix} i\alpha & \beta \\ 
\bar{\beta}  & -i\alpha \end{smallmatrix}\right) \;:\;
\alpha \in \bR, \;\beta \in \bC \right\}.
\end{align}
$G$ acts on $B$ by holomorphic automorphisms
\begin{align}
& \rho_g(w) = \frac{aw+b}{\bar{b}w + \bar{a}}, \\
\label{eq:x-gamma-2}
& X_\gamma =  (-\beta w^2 + 2i\alpha w + \bar{\beta}) \, \partial_w. 
\end{align}
That action preserves the hyperbolic symplectic form (scaled to have curvature $-4$, for simplicity)
\begin{equation}
\omega_B = \frac{d\mathrm{re}(w) \wedge d\mathrm{im}(w)}{(1-|w|^2)^2}.
\end{equation}
There is no longer an invariant primitive. The map \eqref{eq:x-gamma-2} can be lifted to the level of functions, compatibly with Poisson brackets:
\begin{equation} \label{eq:h-gamma}
\gamma \longmapsto H_\gamma = \frac{1}{1-|w|^2} \left( \half(1+|w|^2)\alpha - \mathrm{im}(\beta w) \right) 
= \frac{\alpha - \mathrm{im}(\beta w)}{1-|w|^2} - \half \alpha.
\end{equation}
Taking into account the scaling of the metric, and a choice of identification between disc and half-plane model, these formulae correspond to those from Section \ref{sec:affine}, when restricted to the subgroup fixing a point of $\partial_\infty B$. The $G$-action extends to $\bar{B}$, and its restriction to the unit circle $\partial_\infty B = \bR/2\pi\bZ$ replaces our previous use of the action of $\mathit{PSL}_2(\bR)$ on $\bR P^1$. One can rewrite the extension of \eqref{eq:x-gamma-2} as
\begin{equation}
\label{eq:x-gamma}
\bar{X}_\gamma\,|\,\partial_\infty W
= 2\big(\alpha - \mathrm{im}(\beta w)\big) iw \partial_w.
\end{equation}
Comparing this with \eqref{eq:h-gamma} yields the following:

\begin{lemma} 
If $X_\gamma$ points in positive (negative) direction at $w_\infty \in \partial_\infty B$, $H_\gamma(w) \rightarrow +\infty$ (respectively $-\infty$) as $w \rightarrow w_\infty$. If $X_\gamma$ vanishes at $w_\infty$, 
\begin{equation} \label{eq:h-gamma-bound}
| H_\gamma(w) | \lesssim \frac{|w - w_\infty|}{1-|w|} \quad \text{for $w \in B$ close to $w_\infty$.}
\end{equation} \qed
\end{lemma}

We will also need the notion of geodesic germ (or more properly, germ at infinity of a geodesic). By this we mean a half-infinite geodesic ray $\delta \subset B$, with the understanding that two such are considered equivalent if they differ only by a bounded piece. Any geodesic germ has a unique point at infinity $\lambda = \partial_\infty \delta \in \partial_\infty B$. The behaviour of the functions \eqref{eq:h-gamma} along such germs is as follows:

\begin{lemma} \label{th:tangent-germ}
If $X_\gamma$ is tangent to $\delta$, $H_\gamma|\delta = 0$. More generally, if $\bar{X}_\gamma$ vanishes at $\partial_\infty \delta$, $H_\gamma|\delta$ is bounded.
\end{lemma}

\begin{proof}
Consider the first case. Since everything is invariant under the $G$-action, it suffices to consider the case when $\delta$ is part of the real axis. The tangency assumption implies that $\alpha = 0$, $\beta \in \bR$ in \eqref{eq:x-gamma}, and the desired property can then be read off from \eqref{eq:h-gamma}. The second part follows from \eqref{eq:h-gamma-bound} (a more specific analysis would show that $H_\gamma|\delta \rightarrow 0$ as we approach $\partial_\infty \delta$).
\end{proof}

\subsection{Flat connections and geodesic germs\label{subsec:conn-and-germs}}
Let $S$ be a connected Riemann surface with boundary, equipped with a pair $(A,\delta)$, where
\begin{equation}
A \in \Omega^1(S, \frakg), \quad \delta = (\delta_z)_{z \in \partial S} \text{ is a family of geodesic germs.}
\end{equation}
There is an associated function $\lambda  = \partial_\infty\delta \in \smooth(\partial S,\partial_\infty B)$. We require that $(A,\lambda)$ should satisfy \eqref{eq:new-a-1} and \eqref{eq:new-a-2}. Note that parallel transport is not required to preserve $\delta$.

As in the analogous situation of Section \ref{subsec:package}, the connection $d-A$ induces one-forms $X_A$ and $H_A$, with values in Hamiltonian vector fields and functions; as well as a symplectic form $\omega_A$ and complex structure $J_A$ on the trivial fibre bundle 
\begin{equation} \label{eq:b-trivial}
S \times B \longrightarrow S.
\end{equation}
$X_A$ as well as $J_A$ extend to $S \times \bar{B}$. Finally, $\delta$ determines a germ of a submanifold $\Delta \subset \partial S \times B$, with smooth closure $\bar\Delta \subset \partial S \times \bar{B}$. While $\Delta$ is totally real with respect to $J_A$, it is not necessarily isotropic for $\omega_A$. Instead, there is a preferred one-form vanishing in fibre direction,
\begin{equation} \label{eq:dbeta}
\beta_A \in \Omega^1(\Delta), \quad d\beta_A = \omega_A|\Delta.
\end{equation}
Equivalently, one can view $\beta_A$ as a one-form on $\partial S$ with values in the bundle of functions on $\delta_z$. To define it, we choose $\alpha \in \Omega^1(\partial S,\frakg)$ such that the associated parallel transport maps map the $\delta_z$ to each other. This is not unique, but by Lemma \ref{th:tangent-germ}, the restriction of $H_\alpha$ to $\Delta$ is independent of the choice. One then sets
\begin{equation} \label{eq:dbeta-formula}
\beta_A =(H_{\alpha} - H_{A|\partial S} ) | \Delta.
\end{equation}
To see that \eqref{eq:dbeta-formula} is a primitive for $\omega_A$, one can argue as follows. The entire situation is invariant under gauge transformations, hence reduces to the case where $\delta$ is locally constant, where one can set $\alpha = 0$. In local coordinates $z = s+it$ on $S$ for which the boundary is $t = 0$, this means that $\beta_A = -H_A(\partial_s) \mathit{ds}$. The exterior derivative of this agrees with the restriction of \eqref{eq:omega-a} to the product of $\{t = 0\}$ and a geodesic germ.

\begin{lemma} \label{th:beta-bounded}
$\beta_A$ is locally bounded on $\partial S$; by this we mean that, if $\xi$ is a compactly supported tangent vector field on $\partial S$, then $\beta_A(\xi)$ is a bounded function.
\end{lemma}

\begin{proof}
Again, after gauge transformations, it suffices to consider locally constant $\delta$, where the statement reduces to Lemma \ref{th:tangent-germ}.
\end{proof}

\subsection{Maps to the disc\label{subsec:maps-2}}
Let $S$ and $(A,\delta)$ be as before. We consider maps $u: S \rightarrow B$ satisfying
\begin{equation} \label{eq:u-map-2}
\left\{
\begin{aligned}
& (Du-X_A)^{0,1} = 0, \\
& u(z) \in \delta_z \quad \text{for $z \in \partial S$.}
\end{aligned}
\right.
\end{equation}
These can also be viewed as $J_A$-holomorphic sections $v = (z,u(z))$ of \eqref{eq:b-trivial} with totally real boundary conditions in $\Delta$. There are two versions of energy for solutions of \eqref{eq:u-map-2},
\begin{align}
& E^{\mathit{geom}}(u) = \int_S \|Du-X_A\|^2_B = \int_S v^*\omega_A, \\
\label{eq:e-top}
& E^{\mathit{top}}(u) = \int_S v^*\omega_A - \int_{\partial S} v^*\beta_A,
\end{align}
of which the first one is always nonnegative, whereas the second one is ``topological'' by \eqref{eq:dbeta}. By Lemma \ref{th:beta-bounded}, the difference between the two energies is locally bounded on $S$. To explain the importance of that, let's briefly return to the toy case when $S$ is compact. Then, since there is a unique homotopy class of sections of $(S \times B, \Delta)$, $E^{\mathit{top}}(u)$ is the same for all $u$; and that leads to an upper bound for $E^{\mathit{geom}}(u)$.

As usual, we can apply gauge transformations to \eqref{eq:u-map-2}. In particular, locally near an interior point of $S$, the study of solutions reduces to that of holomorphic functions. The situation at boundary points is more complicated. After a local gauge transformation making $A^\dag = 0$, the boundary condition $\delta^\dag$ consists of a family of geodesic germs which share the same point at infinity. Let's temporarily switch back to the half-plane model for the target space, and assume that the shared point is $\infty$, so that
\begin{equation}
\delta_z^\dag = \{\mathrm{re}(w) = \gamma_z^\dag, \;\; \mathrm{im}(w) \gg 0\} \subset W
\end{equation}
for some $\gamma \in \smooth(\partial S, \bR)$. Then, the equation has the form
\begin{equation} \label{eq:holo-2}
\left\{
\begin{aligned}
& u^\dag: S \longrightarrow W, \\
& \bar\partial u^\dag = 0, \\
& \mathrm{re}(u^\dag(z)) = \gamma_z^\dag \quad \text{for $z \in \partial S$}.
\end{aligned}
\right.
\end{equation}
Solutions of such equations can be produced by means of classical complex analysis. For instance, take $\gamma^\dag \in \smooth_c(\bR,\bR)$. Then the Schwarz integral formula
\begin{equation} \label{eq:integral}
u^\dag(z) = \frac{i}{\pi} \int_\bR \gamma^\dag_\zeta \frac{d\zeta}{z-\zeta}
\end{equation}
defines a holomorphic function on the upper half plane, such that $\mathrm{re}(u^\dag) = \gamma^\dag$ along the real line. Other functions with the same property can be produced from this by adding holomorphic functions with boundary conditions in $i\bR$.

\begin{lemma} \label{th:hofer}
Let $u_k: S \rightarrow B$ be a sequence of solutions of \eqref{eq:u-map-2}. Then, on each fixed compact subset of $S$, $|Du_k|$ (measured with respect to the Euclidean metric on $B$) is bounded.
\end{lemma}

\begin{proof}
We borrow an argument from pseudo-holomorphic curve theory. Suppose that on the contrary, after passing to a subsequence of the $(u_k)$, one has $z_k \rightarrow z_\infty \in S$ such that $|Du_k(z_k)| \rightarrow \infty$. Using Hofer's Lemma (see e.g.\ \cite[p.~137]{abbas14} for an exposition), one finds a sequence of rescalings whose limit is one of the following:
\begin{align}
&
\parbox{36em}{A non-constant holomorphic map from the complex plane to $\bar{B}$.}
\\ &
\parbox{36em}{A non-constant holomorphic map from the upper half plane to $\bar{B}$, with boundary values on a geodesic.}
\end{align}
Of course, neither is possible, which establishes our argument.
\end{proof}

\begin{lemma} \label{th:compactness-copy}
Let $u_k: S \rightarrow B$ be a sequence of solutions of \eqref{eq:u-map-2}. Suppose that there are points $z_k$ contained in a compact subset of $S$, such that $u_k(z_k) \rightarrow \partial_\infty B$. Then $u_k \rightarrow \partial_\infty B$ uniformly on compact subsets. Moreover, a subsequence converges (in the same sense) to a map $u_\infty: S \rightarrow \partial_\infty B$ which satisfies
\begin{equation} \label{eq:infinity}
\left\{
\begin{aligned}
& Du_\infty = \bar{X}_A, \\
& u_\infty(z) = \lambda_z = \partial_\infty \delta_z \quad \text{for $z \in \partial S$.}
\end{aligned}
\right.
\end{equation}
\end{lemma}

\begin{proof}
For a subsequence, Lemma \ref{th:hofer} establishes convergence on compact subsets to some solution $u_\infty: S \rightarrow \bar{B}$, with $u_\infty(z_\infty) \in \partial_\infty B$ for some $z_\infty \in S$. If $z_\infty$ is an interior point, one can argue as in Lemma \ref{th:compactness-0} to conclude that $u_\infty$ takes values in $\partial_\infty B$, and satisfies \eqref{eq:infinity}.

Suppose now that $z_\infty \in \partial S$, in which case necessarily $u_\infty(z_\infty) = \gamma_{z_\infty} = \partial_\infty \delta_{z_\infty}$. We restrict to a half-disc $C$ surrounding $z_\infty$, and apply a local gauge transformation to reduce to $A^\dag = 0$. The outcome is that we have a sequence $u_k^\dag: C \rightarrow W$ of solutions of \eqref{eq:holo-2}, which converge to some $u_\infty^\dag: C \rightarrow \bar{W}$ such that $u_\infty^\dag(0) = \infty$. Without loss of generality, we can assume that the boundary conditions $\gamma^\dag \in \smooth(\partial C,\bR)$ extend to a compactly supported function on the real line. We can then use \eqref{eq:integral} to write
\begin{equation}
u_k^\dag = u^\dag + q_k,
\end{equation}
where $u^\dag$ is a fixed solution, and the $q_k$ are holomorphic functions with boundary values in $i\bR$. In the limit,
\begin{equation}
u_\infty^\dag = u^\dag + q_\infty,
\end{equation}
where $q_\infty: (C,\partial C) \rightarrow (\bar{\bC}, i\bR \cup \{\infty\})$ satisfies $q_\infty(0) = \infty$. If $q_\infty$ is not constant, there are points close to $z = 0$ where $\mathrm{im}(q_\infty)$ has arbitrarily large negative imaginary part, which is a contradiction to the fact that $u_\infty^\dag$ takes values in $\bar{W}$. Hence, $q_\infty$ must be constant equal to $\infty$, which means that $u_\infty^\dag = \infty$, showing that the limit $u_\infty$ takes values in $\partial_\infty B$ and satisfies \eqref{eq:infinity}. As usual, the fact that this holds for subsequences implies $u_k \rightarrow \partial_\infty B$ for the original sequence.
\end{proof}

Finally, we want to consider the special case of the cylinder. Namely, take $A = a_t \mathit{dt} \in \scrP_\tau(S^1)$ for some $\tau>2$, and pull it back to $S = (0,l) \times S^1$. The resulting special case of \eqref{eq:u-map-2} is
\begin{equation} \label{eq:cyl-equation}
\partial_s u + i(\partial_t u - X_{a_t}) = 0,
\end{equation}
and \cite[Lemma 8.4]{seidel17} says the following:

\begin{lemma} \label{th:length}
Solutions of \eqref{eq:cyl-equation} can only exist on a cylinder of length $l \leq L$, where
$L = \pi (2\log(\tau/2 + \sqrt{\tau^2/4-1}))^{-1}$.
\end{lemma}

\section{The Fukaya category\label{sec:fukaya}}
We now proceed to the definition of Fukaya category that arises from the elementary geometric considerations in Section \ref{sec:affine}. Most  of the construction copies the standard pattern. After setting up the geometric framework, we will therefore concentrate on one aspect which is specific to this context, namely how pseudo-holomorphic curves are prevented from escaping to infinity.

\subsection{Target space geometry\label{subsec:space}}
As the fibre, we fix an exact symplectic manifold with boundary. By this, we mean a compact manifold with boundary $M$, together with a symplectic form $\omega_M$, a primitive $\theta_M$, and a compatible almost complex structure $J_M$ which is weakly convex:
\begin{equation} \label{eq:weakly-convex}
\parbox{36em}{Any $J_M$-holomorphic curve touching $\partial M$ must be entirely contained in it.}
\end{equation}
%
The actual target space will be a manifold with boundary $E$, again with a symplectic form $\omega_E$ and primitive $\theta_E$, and which comes with a proper map to the upper half-plane,
\begin{equation} \label{eq:pi}
\pi: E \longrightarrow W.
\end{equation}
For $x \in E$, let $\mathit{TE}_x^h \subset \mathit{TE}_x$ be the symplectic orthogonal complement of $\mathit{TE}_x^v =\mathit{ker}(D\pi_x)$. Let's say that $\pi$ is symplectically locally trivial at $x$ if:
\begin{align}
&
\parbox[t]{36em}{$x$ is a regular point of $\pi$.}
\\
&
\parbox[t]{36em}{$\mathit{TE}_x = \mathit{TE}_x^v \oplus \mathit{TE}_x^h$, which means that both subspaces are symplectic.}
\\
&
\parbox[t]{36em}{$D\pi_x: \mathit{TE}_x^h \rightarrow \mathit{TW}_{\pi(x)}$ pulls back $\omega_W$ to the restriction of $\omega_E$; and the same is true in a neighbourhood of $x$.}
\end{align}
As the name suggests, one can use $\mathit{TE}^h$ as a connection near $x$, so as to locally identify $E$ with a product of the fibre and base (carrying their respective symplectic structures). With this terminology at hand, we require the following conditions:
\begin{align}
& \label{eq:fibrewise-boundary}
\parbox[t]{36em}{
At any point $x \in \partial E$, $\pi$ is symplectically locally trivial; moreover, $\mathit{TE}_x^h \subset T(\partial E)_x$.} 
\\
& \label{eq:at-infinity}
\parbox[t]{36em}{
$\pi$ is symplectically locally trivial outside the preimage of a closed disc $V \subset W$.}
\\
& \label{eq:reference-fibre}
\parbox[t]{36em}{
For some base point $\ast \in W$ lying outside the previously mentioned $V$, the fibre $E_\ast$ is identified with $M$, in a way which is compatible with the symplectic form and its primitive.}
\end{align}

From \eqref{eq:at-infinity}, one sees that $\pi$ is locally a product over the annulus $W \setminus V$. More precisely, let $U \subset W \setminus V$ be an open disc. Then there is a diffeomorphism
\begin{equation} \label{eq:horizontal-collar}
\xymatrix{
\ar[dr]_{\pi} \pi^{-1}(U) \ar[rr]^-{\iso} && U \times M \ar[dl]^-{\;\;\text{projection}} \\ & U
}
\end{equation}
which takes $\omega_E$ to $\omega_W + \omega_M$, and $\theta_E$ to $\theta_W + \theta_M + \{\text{some exact one-form}\}$. For the statement concerning primitives, suppose first that $\ast \in U$. Then, there is a preferred choice of \eqref{eq:horizontal-collar} which restricts to the given identification $E_\ast \iso M$. As consequence, the difference between $\theta_E|\pi^{-1}(U)$ and the pullback of $\theta_W|U + \theta_M$ is a closed one-form vanishing on the fibre over $\ast$, which is therefore exact. To reduce the general case to this, it is enough to observe that outside $V$, the parallel transport maps for the connection $\mathit{TE}^h$ yield exact symplectic isomorphisms between fibres. 

As another consequence of \eqref{eq:horizontal-collar}, one can construct a preferred compactification
\begin{equation} \label{eq:compactification}
\bar{\pi}: \bar{E} \longrightarrow \bar{W}.
\end{equation}
We write $\partial_\infty E = \bar{E} \setminus E = \bar\pi^{-1}(\partial_\infty W)$. The given $\omega_E$ does not extend to $\bar{E}$, but one can define symplectic forms on that space as follows. Take $\psi \in \smooth(\bar{W},\bR)$ which vanishes on $V$ and is equal to $1$ near the boundary; as well as a positive two-form $\omega_{\bar{W}}$ on the closed disc. Then, there is a unique symplectic form $\omega_{\bar{E}}$ on the compactification such that
\begin{equation} \label{eq:modified}
\omega_{\bar{E}}|E = \omega_E + \pi^*(\psi(\omega_{\bar{W}} - \omega_W)).
\end{equation}
Note that $\omega_{\bar{E}}$ is again exact (its restriction to $E$ is cohomologous to $\omega_E$, and the restriction map $H^2(\bar{E};\bR) \rightarrow H^2(E;\bR)$ is of course an isomorphism).

From \eqref{eq:fibrewise-boundary} it follows that $\pi|\partial E$ is a smooth fibre bundle. In fact, by integrating the connection $\mathit{TE}^h$ near $\partial E$, one obtains a diffeomorphism (extending part of the identification $E_* \iso M$)
\begin{equation} \label{eq:vertical-collar}
\xymatrix{
\ar@{^{(}->}[d] \{\text{neighbourhood of $\partial E$}\} \ar[rr]^{\iso} 
&& W \times \{\text{neighbourhood of $\partial M$}\}
\ar@{^{(}->}[d]
\\ 
E \ar[r]_{\pi} & W & W \times M. 
\ar[l]^{\text{projection}}
}
\end{equation}
This takes $\omega_E$ to $\omega_W + \omega_M$, and $\theta_E$ to $\theta_W +\theta_M + \{\text{some exact one-form}\}$. The fact that we can take the image of \eqref{eq:vertical-collar} to be of the form $W \times \{\text{neighbourhood of $\partial M$}\}$, rather than just some neighbourhood of $W \times \partial M$, depends on \eqref{eq:horizontal-collar}.

Our next task is to introduce the relevant class of Lagrangian submanifolds. Those are connected $L \subset E$ such that:
\begin{align}
& \label{eq:lag}
\parbox[t]{36em}{$L$ is disjoint from $\partial E$.} 
\\ & \label{eq:lag-2}
\parbox[t]{36em}{$\pi|L$ is proper, and there is a $\lambda_L \in \bR$ such that $\pi(L)$ is contained in the union of a compact set and the vertical half-open segment $\{\mathrm{re}(w) = \lambda_L, \, \mathrm{im}(w) \ll 1\} \subset W$.}
\\ & \label{eq:lag-3}
\parbox[t]{36em}{$L$ is exact with respect to $\theta_E$.}
\end{align}
On a suitable segment as in \eqref{eq:lag-2}, $L$ is given by a family of exact Lagrangian submanifolds in the fibres, which are mapped to each other by $\mathit{TE}^h$-parallel transport. As a consequence, the closure $\bar{L} \subset \bar{E}$ is a smooth submanifold with boundary $\partial\bar{L} = \bar{L} \cap \bar{E}_{\lambda_L}$. If one chooses the function $\psi$ in \eqref{eq:modified} to be zero on a sufficiently large subset, $\bar{L}$ will be Lagrangian with respect to $\omega_{\bar{E}}$. Note that we have not made any assumption on the local structure of the critical points of \eqref{eq:pi}; if we did impose Lefschetz (complex nondegeneracy) conditions, then the Lefschetz thimbles, for paths that become vertical segments at infinity, would belong to the class under consideration.

Define $\scrJ(E)$ to be the space of compatible almost complex structures $J$ on $E$ with the following properties:
\begin{align}
& 
\parbox[t]{36em}{The image of $J$ under \eqref{eq:vertical-collar} (possibly after shrinking the neighbourhoods that appear there) is the product of $J_M$ and the complex structure of the base.}
\\
&
\parbox[t]{36em}{Outside the preimage of some compact subset of $W$, $D\pi$ is $J$-holomorphic.}
\\
&
\parbox[t]{36em}{$J$ extends smoothly to an almost complex structure $\bar{J}$ on $\bar{E}$.}
\end{align}
This extension automatically has the property that $D\bar\pi$ is $\bar{J}$-holomorphic at any point of $\partial_\infty E$. Moreover, the closure $\bar{L}$ of any of our Lagrangian submanifolds is $\bar{J}$-totally real. Given any $J$, one can arrange the choice of $\psi$ in \eqref{eq:modified} so that $\bar{J}$ is compatible with $\omega_{\bar{E}}$.

Any $J \in \scrJ(E)$, combined with $\omega_E$, determines a Riemannian metric, whose norm we will denote by $\|\cdot\|_{E,J}$. Given two such almost complex structures, the associated metrics are commensurable (each bounds the other up to a constant). To see that, note that at any point $x$ close to infinity, the subspaces $TE_x^h$ and $TE_x^v$ are orthogonal; the metric on the first summand is the pullback by projection of the hyperbolic metric on $TW_{\pi(x)}$, while the commensurability class on the second summand is governed by the fact that it extends to $\bar{E}$. With that in mind, we will sometimes write $\|\cdot\|_E$ if only the commensurability class of the metric is important. By the same argument, if we have any Riemannian metric on $\bar{E}$, there is an inequality
\begin{equation} \label{eq:compare-norms}
\|X\|_{\bar{E}} \lesssim \|X\|_E  \quad \text{for any $X \in TE$.}
\end{equation}

For any $\gamma \in \frakg_{\mathit{aff}}$, we consider the class $\scrH_\gamma(E)$ of those $H \in \smooth(E,\bR)$ such that:
\begin{equation}
\parbox[t]{36em}{
Outside a compact subset of $E \setminus \partial E$ (which means, near $\partial E$ as well as outside a preimage of some compact subset of $W$), $H$ is the pullback of $H_\gamma$; see \eqref{eq:ha}.}
\end{equation}
On the region where these restrictions apply and where the fibration is symplectically locally trivial, the Hamiltonian vector field $X$ of $H$ agrees with the unique lift of $X_\gamma \in \smooth(W,TW)$ to $\mathit{TE}^h$. In particular, $X$ is tangent to $\partial E$. Moreover, it extends to a vector field $\bar{X}$ on $\bar{E}$, which is tangent to $\partial_\infty E$.

\subsection{Energy}
Let $S$ be a connected oriented Riemann surface with boundary, together with a pair $(A,\lambda)$ satisfying \eqref{eq:a-lambda-1} and \eqref{eq:a-lambda-2}. We equip this with the following additional data:
\begin{align}
& \label{eq:data-1}
\parbox[t]{36em}{
A family of almost complex structures $J = (J_z)$, $J_z \in \scrJ(E)$, parametrized by $z \in S$.}
\\ & \label{eq:data-2}
\parbox[t]{36em}{
A one-form $K \in \Omega^1(S,\smooth(E,\bR))$ with values in functions on $E$ (equivalently, a section of the pullback bundle $T^*S \rightarrow S \times E$; or, a one-form on $S \times E$ which vanishes in $TE$-direction), such that for each $\xi \in TS$, $K(\xi)$ lies in $\scrH_{A(\xi)}(E)$. Let $X_K \in \Omega^1(S,\smooth(E,TE))$ be the associated one-form with values in Hamiltonian vector fields (or, section of $\mathit{Hom}(TS,TE) \rightarrow S \times E$).}
\\ & \label{eq:data-3}
\parbox[t]{36em}{A family of Lagrangian submanifolds $L_z$ parametrized by $z \in \partial S$ (equivalently, a subbundle $\Lambda \subset \partial S \times E$ whose fibres are Lagrangian submanifolds), which lie in the general class \eqref{eq:lag}--\eqref{eq:lag-3}, with $\lambda_{L_z} = \lambda_z$. 
}
\end{align}
We consider maps $u: S \rightarrow E$ such that:
\begin{equation} \label{eq:main-cr}
\left\{
\begin{aligned}
& (Du-X_K)^{0,1} = 0 \quad \text{with respect to $J_{z,u(z)}$,} \\
& u(z) \in L_z \quad \text{for $z \in \partial S$.}
\end{aligned}
\right.
\end{equation}
The geometric energy of a solution is
\begin{equation}
\label{eq:e-geom}
E^{\mathit{geom}}(u) = \int_S \|Du - X_K\|_{E,J_z}^2.
\end{equation}

One can approach \eqref{eq:main-cr} in geometric terms resembling those from Sections \ref{subsec:package} and \ref{subsec:conn-and-germs}. Think of $X_K$ as a Hamiltonian connection on the (trivial) fibre bundle
\begin{equation} \label{eq:e-trivial}
S \times E \longrightarrow S,
\end{equation}
which lifts any vector field $\xi$ on $S$ to $\xi + X_K(\xi)$. The curvature of this connection is a two-form $R_K \in \Omega^2(S,\smooth(E,\bR))$ (equivalently, a section of $\Lambda^2T^*S \rightarrow S \times E$; or, a two-form on $S \times E$ which vanishes if we insert an element of $TE$), given in local coordinates $z = s+it$ on $S$ by
\begin{equation}
R_K = \big( \partial_t K(\partial_s) - \partial_s K(\partial_t) + \{K(\partial_s), K(\partial_t)\} \big)\, \mathit{ds} \wedge \mathit{dt}.
\end{equation}

\begin{lemma} \label{th:rk-compact}
$R_K$ takes values in functions that vanish outside a compact subset of $E \setminus \partial E$.
\end{lemma}

\begin{proof}
This follows from the assumption that $K(\xi) \in \scrH_{A(\xi)}(E)$, and the flatness of $A$.
\end{proof}

The connection determines two-forms $\omega_K^{\mathit{geom}}$, $\omega_K^{\mathit{top}}$ on $E$, which agree with $\omega_E$ on each fibre: in local coordinates on $S$ as before, these can be written as
\begin{align}
& \omega_K^{\mathit{geom}} = \omega_E + \omega_E(X_K(\partial_s),\cdot) \wedge \mathit{ds}
+ \omega_E(X_K(\partial_t),\cdot) \wedge \mathit{dt} - \omega_E(X_K(\partial_s), X_K(\partial_t)) \mathit{ds} \wedge \mathit{dt}, \\
&
\omega_K^{\mathit{top}}  = \omega_E - d( K(\partial_s) \mathit{ds}) - d (K(\partial_t) \mathit{dt})
= \omega_K^{\mathit{geom}} + R_K.
\end{align}
The second one is closed, and has an obvious primitive,
\begin{equation}
\theta_K^{\mathit{top}} = \theta_E - K.
\end{equation}
There is a one-form $\beta_K$ on $\Lambda$ vanishing in fibre direction, and a function $P_K$, such that:
\begin{align}
\label{eq:beta-k}
& \beta_K \in \Omega^1(\Lambda), \quad d\beta_K = \omega_K^{\mathit{top}}|\Lambda, \\
\label{eq:p-k}
& P_K \in \smooth(\Lambda, \bR), \quad dP_K + \beta_K = \theta_K^{\mathit{top}}|\Lambda.
\end{align}
To see that, one argues as follows. By the exactness assumption, there are functions on each $L_z$ whose derivative is $\theta_E|L_z$. We assemble those into a single function $P_K$ on $\Lambda$ (which is unique up to adding locally constant functions). Clearly, $\theta_K^{\mathit{top}}|\Lambda - dP_K$ then vanishes in fibre direction, and one defines this to be $\beta_K$. For closer resemblance with \eqref{eq:dbeta-formula}, let's note the following. Since all the $L_z$ are exact, their $z$-dependence is Hamiltonian, hence can be expressed by a one-form $\alpha$ vanishing in fibre direction,
\begin{equation}
\alpha \in \Omega^1(\Lambda), \quad d\alpha = \omega_E|\Lambda.
\end{equation}
It follows that $\beta_K - \alpha + K|\Lambda$ is a closed one-form which vanishes in fibre direction. Since the $L_z$ are connected by assumption, we then necessarily have
\begin{equation} \label{eq:beta-k-2}
\beta_K = \alpha - K|\Lambda + \{\text{some one-form pulled back from $\partial S$}\}.
\end{equation}

\begin{lemma} \label{th:locally-bounded-2}
$\beta_K$ is locally bounded on $\partial S$.
\end{lemma}

\begin{proof}
Note that this is a statement about the behaviour near $\partial_\infty L_z = \bar{L}_z \cap \partial_\infty E$, where the geometry is governed by \eqref{eq:vertical-collar}. This allows us to reduce considerations to the case of a product fibration $E = W \times M$, and where
\begin{equation}
L_z = \{\mathrm{re}(w) = \lambda_z, \;\; 0 < \mathrm{im}(w) \ll 1\} \times L_{M,z},
\end{equation}
for some family of closed exact Lagrangian submanifolds $L_{M,z} \subset M \setminus \partial M$, with corresponding functions $\theta_M|L_{M,z} = dP_{M,z}$; and where $K$ is just given by the pullback of $H_A$. As a consequence, one can write \eqref{eq:beta-k-2} as the sum of two terms, one being \eqref{eq:a-primitive}, and the other coming from the fibre $M$. As observed in Section \ref{sec:affine}, the first term vanishes, leaving the fibre contribution, which is independent of the $W$-direction, hence necessarily bounded. 
\end{proof}

The connection also determines an almost complex structure $J_K$ on $S \times E$, which is such that projection to $S$ is pseudo-holomorphic, and $\Lambda$ a totally real submanifold. Both the connection and $J_K$ extend to $\bar{E}$, and $\bar\Lambda \subset \bar{E} \times S$ is a submanifold with boundary. Returning to our main topic, solutions of \eqref{eq:main-cr} can be viewed as $J_K$-holomorphic sections $v = (z,u(z))$ of \eqref{eq:e-trivial}, with boundary conditions given by $\Lambda$. One can rewrite the geometric energy \eqref{eq:e-geom} in these terms, and also introduce its topological cousin:
\begin{align} 
\label{eq:e-geom-2}
&
E^{\mathit{geom}}(u) = \int_S v^*\omega_K^{\mathit{geom}}, \\
\label{eq:e-top-2}
&
E^{\mathit{top}}(u) = \int_S v^*\omega_K^{\mathit{top}} - \int_{\partial S} v^*\beta_K.
\end{align}
Clearly, the relation between the two is that 
\begin{equation} \label{eq:e-e}
E^{\mathit{geom}}(u) = E^{\mathit{top}}(u) - \int_S v^*R_K + \int_{\partial S} v^*\beta_K.
\end{equation}

\begin{example} \label{th:s-is-compact-2}
In the toy model case where $S$ is compact, we have
\begin{equation}
E^{\mathit{top}}(u) = \int_S v^*d\theta_K^{\mathit{top}} - \int_{\partial S} v^*\theta_K^{\mathit{top}} +
\int_{\partial S} v^*dP_K = 0.
\end{equation}
Because $R_K$ vanishes outside a compact subset of $E \setminus \partial E$, and the last term in \eqref{eq:e-e} is bounded by Lemma \ref{th:locally-bounded-2}, we get a bound on the geometric energy of solutions.
\end{example}

\subsection{Local compactness} 
We will now consider ``containment methods'' which keep solutions of \eqref{eq:main-cr} from either reaching the boundary (in fibre direction), or going to infinity (over the base); the arguments for the latter and more important issue are modelled on those in Sections \ref{subsec:maps} and \ref{subsec:maps-2}.

\begin{lemma} \label{th:convex}
Suppose that $\partial S \neq \emptyset$. Then, no solution of \eqref{eq:main-cr} can reach $\partial E$.
\end{lemma}

\begin{proof}
Whenever $u(z)$ is close to $\partial E$ (in which case $z$ is necessarily an interior point), we can use \eqref{eq:vertical-collar} to project it to a map to $M$, which is $J_M$-holomorphic. By \eqref{eq:weakly-convex}, it follows that if $z$ intersects $\partial E$, it must be entirely contained in it, which contradicts the boundary condition.
\end{proof}

\begin{lemma} \label{th:gromov-1}
Let $u_k$ be a sequence of solution of \eqref{eq:main-cr}, such that on each relatively compact open subset $T \subset S$, the energy $E^{\mathit{geom}}(u_k|T)$ is bounded. Then the pointwise norm $\|(Du_k - X_K)|T\|_{\bar{E}}$ is also bounded.
\end{lemma}

\begin{proof}
By \eqref{eq:compare-norms}, we get a bound on $\int_{T} \|Du_k - X_K\|^2_{\bar{E}}$. Since the vector fields $X_K$ extend to $\bar{E}$, they are bounded in any metric there, so the outcome is that we have a bound on $\int_T \| Du_k\|^2_{\bar{E}}$. Suppose that we have a sequence $z_k \rightarrow z_\infty \in S$, for which $\|D u_k(z_k)\|_{\bar{E}}$ goes to $\infty$. The same rescaling argument as in the proof of Lemma \ref{th:hofer} would then lead to one of the following:
\begin{align}
& \label{eq:no-sphere}
\parbox{36em}{A non-constant $\bar{J}_{z_\infty}$-holomorphic map $u: \bC \rightarrow \bar{E}$, with $\int \|Du\|^2_{\bar{E}} < \infty$.}
\\ & \label{eq:no-disc}
\parbox{36em}{Assuming $z_\infty \in \partial S$: a non-constant map from the upper half-plane to $\bar{E}$, with boundary conditions on $\bar{L}_{z_\infty}$, with the same properties as before.
}
\end{align}
Recall that $\bar{E}$ carries a compatible symplectic form as in \eqref{eq:modified}. We can use removal of singularities for pseudo-holomorphic maps, and the exactness of that form, to rule out \eqref{eq:no-sphere}. The same applies to \eqref{eq:no-disc}, since the relative class $[\omega_{\bar{E}}] \in H^2(\bar{E},\bar{L}_{z_\infty};\bR)$ is also zero, as restriction to $E$ shows.
\end{proof}

\begin{remark}
Even though this is not necessary for our purpose, it may be of interest to note that one can upgrade the bound in Lemma \ref{th:gromov-1} from $\|\cdot\|_{\bar{E}}$ to the stronger norm $\|\cdot\|_E$. Namely, suppose the opposite is true, meaning that we have a sequence $z_k \rightarrow z_\infty$, for which $\|(Du_k - X_K)(z_k)\|_E$ goes to $\infty$. Since we already have a bound on $\|Du_k(z_k) - X_K\|_{\bar{E}}$, $u_k(z_k)$ must go to $\partial_\infty E$. It also follows that there is a neighbourhood $T$ of $z_\infty$ such that for all $k \gg 0$, $u_k|T$ lies close to $\partial_\infty E$. Then, $\pi(u_k|T)$ is a solution of an equation \eqref{eq:u-equation}. From Lemma \ref{th:schwarz} we get explicit bounds on $\|D\pi(u_k - X_K)\|_W$ at any point of $T$. These can be also thought of as bounds on the $\mathit{TE}^h$ component of $u_k - X_K$. Since the $\mathit{TE}^v$ component is bounded by our previous argument, we obtain a contradiction.
\end{remark}

\begin{lemma} \label{th:gromov-2}
Let $u_k$ be as in Lemma \ref{th:gromov-1}. Suppose that there is a sequence of points $z_k$, contained in a compact subset of $S$, such that $u_k(z_k) \rightarrow \partial_\infty E$. Then $u_k \rightarrow \partial_\infty E$ uniformly on compact subsets. Moreover, a subsequence converges (in the same sense) to a map $u_\infty: S \rightarrow \partial_\infty E$ such that
\begin{equation} \label{eq:proj-boundary}
\left\{
\begin{aligned}
& D(\pi(u_\infty)) = \bar{X}_A, \\
& \pi(u_\infty(z)) = \lambda_z \quad \text{for $z \in \partial S$.}
\end{aligned}
\right.
\end{equation}
\end{lemma}

\begin{proof}
Convergence of a subsequence follows from Lemma \ref{th:gromov-1}. The limit $u_\infty$ satisfies the same equation as in \eqref{eq:main-cr} for the extended data $\bar{J}_z$ and $\bar{L}_z$. By assumption, there is a point $z_\infty$ such that $u_\infty(z_\infty) \in \partial_\infty E$. Near that point, $\pi(u_\infty): S \rightarrow \bar{W}$ is a solution of \eqref{eq:u-equation}. By the same argument as in Lemma \ref{th:compactness-0}, this implies that $u_\infty^{-1}(\partial_\infty E)$ is open and closed, hence all of $S$; it also follows that $\pi(u_\infty)$ satisfies \eqref{eq:proj-boundary}. By applying the same argument to subsequences, we get convergence $u_k \rightarrow \partial_\infty E$ for the original sequence. 
\end{proof}


\subsection{Strip-like ends}
Suppose that we are given $(A = a_t \mathit{dt}, \lambda_0, \lambda_1) \in \scrP_{\mathit{aff}}([0,1])$. Additionally, take a time-dependent function $H = (H_t)$ with $H_t \in \scrH_{a_t}(E)$, and two Lagrangian submanifolds $(L_0,L_1)$ whose behaviour at infinity \eqref{eq:lag-2} satisfies $\lambda_{L_k} = \lambda_k$. Let $\phi$ be the time-one map of the time-dependent Hamiltonian vector field $X_t$ of $H_t$, and set
\begin{equation}
L_{1}^\dag = \phi^{-1}(L_{1}). 
\end{equation}
Then, $\lambda_{L_1^\dag} = \lambda_1^\dag$, in the notation of \eqref{eq:to-the-right}. It follows from the definition of $\scrP_{\mathit{aff}}([0,1])$ that $L_0 \cap L_1^\dag$ must be compact. We additionally assume the following (which is true for generic choice of Hamiltonian):
\begin{equation} \label{eq:transverse-intersection}
\parbox{36em}{$L_0 \cap L_1^\dag$ is transverse.}
\end{equation}
Let $\scrC(H,L_0,L_1)$ be the set of $X_t$-chords connecting our Lagrangian submanifolds:
\begin{equation}
\left\{
\begin{aligned}
& x: [0,1] \longrightarrow E, \\
& dx/dt = X_t, \\
& x(0) \in L_0, \;\; x(1) \in L_1.
\end{aligned}
\right.
\end{equation}
These correspond bijectively to points $x(0) \in L_0 \cap L_1^\dag$. Hence, under our assumption, $\scrC(H,L_0,L_1)$ is finite. Given functions $P_{L_j} \in \smooth(L_j,\bR)$ with $dP_{L_j} = \theta_E|L_j$, we define the action to be
\begin{equation} \label{eq:action}
A_H(x) = \int_{[0,1]} -x^*\theta_E + H_t(x(t)) \mathit{dt} \, + P_{L_1}(x(1)) - P_{L_0}(x(0)).  
\end{equation}

Take a boundary-punctured disc $S$, with ends \eqref{eq:ends}. Suppose that for each end, we have chosen
$(A_j = a_{j,t} \mathit{dt},\lambda_{j,0}, \lambda_{j,1}) \in \scrP_{\mathit{aff}}([0,1])$, and on the surface itself, an $(A,\lambda) \in \scrP_{\mathit{aff}}(S,\Sigma)$. Additionally we choose, for each end, a pair $(L_{j,0},L_{j,1})$ of Lagrangian submanifolds, whose behaviour at infinity is governed by $(\lambda_{j,0},\lambda_{j,1})$; as well as
\begin{align}
& \label{eq:floer-1}
J_j = (J_{j,t}), \quad J_{j,t} \in \scrJ(E), \\ 
& \label{eq:floer-2}
H_j = (H_{j,t}), \quad H_{j,t} \in \scrH_{a_{j,t}}(E),
\end{align}
where the latter satisfies the transverse intersection condition \eqref{eq:transverse-intersection}. On $S$, we choose $(J,K,L)$ as in \eqref{eq:data-1}--\eqref{eq:data-3}, which are compatible with the choices made over the ends, in the following sense:
\begin{align}
& \label{eq:match-1}
\parbox[t]{36em}{As $s \rightarrow \pm\infty$, $J_{\epsilon_j(s,t)} \rightarrow J_{j,t}$ exponentially fast (in any $C^r$ topology). Moreover, there is a compact subset of $E \setminus \partial E$ such that $J_{\epsilon_j(s,t)} = J_{j,t}$ outside that subset.}
\\
& \label{eq:match-2}
\parbox{36em}{$\epsilon_j^*K = H_{j,t} \mathit{dt}$.}
\\
\label{eq:match-3} &
\parbox{36em}{$L_{\epsilon_j(s,0)} = L_{j,0}$, $L_{\epsilon_j(s,1)} = L_{j,1}$.}
\end{align}

\begin{remark}
We have chosen to impose asymptotic conditions in \eqref{eq:match-1}, rather than strict equality, since that makes transversality arguments easier (compare e.g.\ \cite[Lemma 9.8]{seidel17}), while still allowing for the standard gluing constructions. For a more systematic approach, one could extend that idea to \eqref{eq:match-2} and \eqref{eq:match-3}, relaxing the conditions there to asymptotic ones; but it seems that in practice, nothing would be gained by that.
\end{remark}

Given this, we consider solutions of \eqref{eq:main-cr} with limits
\begin{equation} \label{eq:with-limits}
\textstyle{\lim_{s \rightarrow \pm\infty}} u(\epsilon_j(s,\cdot)) = x_j \in \scrC(H_j,L_{j,0},L_{j,1}).
\end{equation}
One can choose $P_K$, restricted to $\epsilon_j(\cdot,k)$ ($k = 0,1$) to be independent of $s$, which means that it is just given by a primitive $P_{L_{j,k}}$ of $\theta_E|L_{j,k}$. Using those primitives to define the actions \eqref{eq:action}, one then finds that by Stokes, 
\begin{equation} \label{eq:action-energy}
E^{\mathit{top}}(u) = A_{H_0}(x_0) - \sum_{j=1}^d A_{H_j}(x_j).
\end{equation}

\begin{lemma} \label{th:compactness-1}
There is a bound on the geometric energy $E^{\mathit{geom}}(u)$ of solutions of \eqref{eq:main-cr}, \eqref{eq:with-limits}.
\end{lemma}

\begin{proof}
The condition \eqref{eq:match-2} implies that $R_K$ vanishes over the strip-like ends. Together with Lemma \ref{th:rk-compact}, it follows that $R_K$ is a compactly supported two-form on $S$, taking values in functions on $E$ that are bounded. Hence, we get an upper bound on its integral over any section. Similarly, $\beta_K$ vanishes on the ends, which together with Lemma \ref{th:locally-bounded-2} yields a bound on its integral. In view of \eqref{eq:e-e}, the bound on the topological energy from \eqref{eq:action-energy} now implies the desired result.
\end{proof}

\begin{lemma} \label{th:compactness-2}
There is a compact subset of $E \setminus \partial E$ which contains all solutions of \eqref{eq:main-cr}, \eqref{eq:with-limits}.
\end{lemma}

\begin{proof}
Suppose that the opposite is true. Inspection of the proof of Lemma \ref{th:convex} shows that there is a neighbourhood of $\partial E$ which no solution can enter. Hence, we must then have a sequence of solutions $u_k$ and points $z_k \in S$ such that $u_k(z_k) \rightarrow \partial_\infty E$.

If $z_k$ has a convergent subsequence, Lemma \ref{th:gromov-2} implies the existence of a map $u_\infty: S \rightarrow \partial_\infty E$ satisfying \eqref{eq:proj-boundary}. In that case, $v_\infty(s,t) = \pi(u_\infty(\epsilon_j(s,t)))$ (for any choice of $j$) is a map taking values in $\partial_\infty W$, and such that
\begin{equation} \label{eq:x-contradiction}
\left\{
\begin{aligned}
& \partial_s v_\infty = 0, \\
& \partial_t v_\infty = \bar{X}_{a_j(t)}, \\
& v_\infty(s,0) = \lambda_{j,0}, \quad v_\infty(s,1) = \lambda_{j,1}.
\end{aligned}
\right.
\end{equation}
The existence of such a map would mean that $\lambda_{j,0} = \lambda_{j,1}^\dag$, which is a contradiction.

The other possibility is that, after passing to a subsequence, we have $z_k = \epsilon_j(s_j,t_j)$ for some $j$, and where $\pm s_j \rightarrow \infty$. In that case, we can consider the shifted sequence $\tilde{u}_k(s,t) = u_k(\epsilon(s+s_j,t))$. On any compact subset of $\bR \times [0,1]$, these maps (for $j \gg 0$) satisfy equations
\begin{equation}
\left\{
\begin{aligned}
& \partial_s \tilde{u}_k + \tilde{J}_{k,s,t}(\partial_t \tilde{u}_k - X_{j,t}) = 0, \\
& \tilde{u}_k(s,0) \in L_{j,0}, \;\; \tilde{u}_k(s,1) \in L_{j,1}.
\end{aligned}
\right.
\end{equation}
where the almost complex structures $\tilde{J}_{k,s,t}$ converge to $J_{j,t}$ as $k \rightarrow \infty$. One can apply the same argument as in Lemmas \ref{th:gromov-1} and \ref{th:gromov-2} to conclude that a subsequence converges to some $\tilde{u}_\infty$ whose projection to $W$ satisfies the analogue of \eqref{eq:proj-boundary}, hence leads to a contradiction, exactly as in \eqref{eq:x-contradiction}.
\end{proof}

\subsection{Conclusion}
We now explain how the previous considerations enter into the (otherwise standard) definition of the Fukaya category $\scrA = \scrF(\pi)$. For simplicity, we take coefficients in $\bK = \bZ/2$, and introduce no gradings.

Objects of the category are Lagrangian submanifolds $L \subset E$ as in \eqref{eq:lag}--\eqref{eq:lag-3}. Given two such submanifolds $(L_0,L_1)$, whose behaviour at infinity is governed by $(\lambda_{L_0},\lambda_{L_1})$, we choose once and for all some $A_{L_0,L_1} \in \Omega^1([0,1], \frakg_{\mathit{aff}})$ so that $(A_{L_0,L_1},\lambda_{L_0},\lambda_{L_1}) \in \scrP_{\mathit{aff}}([0,1])$. Additionally, choose $J_{L_0,L_1}$ and $H_{L_0,L_1}$ as in \eqref{eq:floer-1}, \eqref{eq:floer-2}, assumed to be generic so as to satisfy transversality requirements. We can then use those to define the Floer cochain complex $(\mathit{CF}^*(L_0,L_1), \mu^1)$ (really, an ungraded $\bK$-vector space together with a differential). Compactness issues are taken care of by the exactness assumptions, together with Lemmas \ref{th:compactness-1} and \ref{th:compactness-2}.

The next step is to define the product on a triple of objects,
\begin{equation} \label{eq:mu2}
\mu^2: \mathit{CF}^*(L_1,L_2) \otimes \mathit{CF}^*(L_0,L_1) \longrightarrow \mathit{CF}^*(L_0,L_2).
\end{equation}
For that, one takes $S$ to be the disc with $3$ boundary punctures. On the ends, we consider
\begin{equation}
(A_j, \lambda_{j,0}, \lambda_{j,1}) = \begin{cases} 
(A_{L_0,L_2}, \lambda_{L_0},\lambda_{L_2}) & j = 0, \\
(A_{L_{j-1},L_j}, \lambda_{L_{j-1}}, \lambda_{L_j}) & j > 0.
\end{cases}
\end{equation}
Take the function $\lambda_{L_0,L_1,L_2} \in \smooth(\partial S, \bR)$ which, on the boundary component $\partial_j S$, is equal to $\lambda_{L_j}$. By Corollary \ref{th:affine-a-space}, there is an $A_{L_0,L_1,L_2} \in \Omega^1(S, \frakg_{\mathit{aff}})$ such that $(A_{L_0,L_1,L_2}, \lambda_{L_0,L_1,L_2}) \in \scrP(S,\Sigma)$. Having done that, one chooses the remaining data $J_{L_0,L_1,L_2}$ and $K_{L_0,L_1,L_2}$ compatibly. Counting solutions of the associated equation \eqref{eq:main-cr}, \eqref{eq:with-limits} then yields \eqref{eq:mu2}

The construction of the higher $A_\infty$-operations $\mu^d$ is parallel. The only additional proviso is that we have to choose all the relevant structures smoothly depending on the moduli of the discs with $(d+1)$ boundary punctures (which is possible since the spaces of choices are always weakly contractible), and so that as one approaches the boundary of that moduli space, they are compatible with the limit in which the discs split into pieces. The basic compactness results (Lemmas \ref{th:gromov-1} and \ref{th:gromov-2}) similarly need to be extended, to accomodate sequences of solutions $u_k: S_k \rightarrow E$ whose domains approach a limit $S_\infty$, or degenerate via neck-stretching.

\section{The closed-open string map\label{sec:fukaya2}}
We transition the construction of the Fukaya category to the more general framework which uses all hyperbolic isometries, and then explain how that naturally incorporates a construction of the closed-open string map as well.

\subsection{Revisiting the definition of the Fukaya category}
We keep the same class of target spaces as in Section \ref{subsec:space}, except that the base will now be thought of as the disc $B$. The almost complex structures will be as before; for the Hamiltonians, we allow classes $\scrH_\gamma(E)$ of functions which, outside a compact subset of $E \setminus \partial E$, agree with the pullback of $H_\gamma$ for any $\gamma \in \frakg$. Most importantly, for the Lagrangian submanifolds $L \subset E$, we now allow the behaviour of $\pi(L)$ outside a compact subset to be given by an arbitrary geodesic germ $\delta_L$. At the same time, every such $L$ should come with a specified lift of the point $\partial_\infty \delta_L \in \partial_\infty B = \bR/2\pi\bZ$ to $\bR$. We denote this lift by $\tilde{\lambda}_L$. 

Our Riemann surfaces will now come with $(A,\tilde{\lambda})$ as in \eqref{eq:new-a-1}, \eqref{eq:new-a-2}. The choices of almost complex structures $J = (J_z)$ remains the same, but the functions $K(\xi)$ now belong to the more general class associated to $A(\xi) \in \frakg$; and similarly, we have more freedom in choosing the Lagrangians $L_z$, $z \in \partial S$. As far as energy considerations for the solutions of the associated equations \eqref{eq:main-cr} are concerned, the formalism remains as before, except that the analogue of Lemma \ref{th:locally-bounded-2} now appeals to Lemma \ref{th:beta-bounded}. The proof of Lemma \ref{th:gromov-1} goes through as before. In Lemma \ref{th:gromov-2}, given that a subsequence of $u_k$ converges to $u_\infty$, with some point $z_\infty \in u_\infty^{-1}(\partial_\infty E)$, one restricts to a neighbourhood of $z_\infty$ and to $k \gg 0$, and applies Lemma \ref{th:compactness-copy} to $\pi_k(u_k)$ in that neighbourhood. 

\subsection{The interior puncture}
Suppose that we have $A = a_t \mathit{dt} \in \scrP_\tau(S^1)$. Choose a corresponding time-dependent Hamiltonian $H = (H_t)$, $H_t \in \scrH_{a_t}(E)$, and consider the set $\scrC(H)$ of one-periodic orbits of its Hamiltonian vector field $X = (X_t)$:
\begin{equation} \label{eq:closed-orbits}
\left\{ 
\begin{aligned}
& x: \bR/\bZ \longrightarrow E, \\
& dx/dt =X_t.
\end{aligned}
\right.
\end{equation}

\begin{lemma}
For a given $H$, all orbits \eqref{eq:closed-orbits} are contained in a compact subset of $E \setminus \partial E$.
\end{lemma}

\begin{proof}
Our vector field admits a smooth extension $\bar{X}$ to $\bar{E}$, which is everywhere tangent to the boundary. If the Lemma were false, there would have to be a sequence of one-periodic orbits $x_k$ converging to some limit, which takes values in $\partial E$ or $\partial_\infty E$.

Outside a compact subset of $E \setminus \partial E$, the vector fields $X_t$ project to the corresponding infinitesimal hyperbolic isometries $X_{a_t}$. Hence, for $k \gg 0$, $\pi(x_k(0))$ would have to be a fixed point of the holonomy of $A$ acting on $B$. By definition, this holonomy is a hyperbolic element, hence acts freely, which is a contradiction.
\end{proof}

Let's assume from now on that all orbits \eqref{eq:closed-orbits} are nondegenerate, and choose a family $J = (J_t)_{t \in \bR/\bZ}$, $J_t \in \scrJ(E)$, of almost complex structures. One can then consider the Hamiltonian Floer equation with limits $x_\pm$ as in \eqref{eq:closed-orbits},
\begin{equation} \label{eq:h-floer}
\left\{
\begin{aligned}
& u: S = \bR \times S^1 \longrightarrow E, \\
& \partial_s u + J_t(\partial_t u - X_t) = 0, \\
& \textstyle{\lim_{s \rightarrow \pm\infty}} u(s,\cdot) = x_{\pm}.
\end{aligned}
\right.
\end{equation}
Locally on $S$, this belongs to the same class we have studied before, with some simplifications (for instance, 
there is no distinction between geometric and topological energy). The argument from Lemma \ref{th:convex} shows that for any solution, $u^{-1}(\partial E)$ is open and closed, and hence (by looking at the limits) must be empty. For preventing solutions from going to $\partial_\infty E$, Lemma \ref{th:gromov-1} and \ref{th:gromov-2} are again the basic ingredients. More precisely, in the application of Lemma \ref{th:gromov-2}, the limiting maps $v_\infty(s,t) = \pi(u_\infty(s,t)): \bR \times S^1 \longrightarrow \bar{B}$ would satisfy
\begin{equation} \label{eq:v-map-2}
\left\{
\begin{aligned}
& \partial_s v_\infty = 0, \\
& \partial_t v_\infty  = \bar{X}_{a_t},
\end{aligned}
\right.
\end{equation}
compare \eqref{eq:x-contradiction}. Since the holonomy $g$ of $A$ is hyperbolic, there are exactly two solutions of \eqref{eq:v-map-2}, which correspond to fixed points of the action of $g$ on $\partial_\infty B$. However, if a sequence $u_k$ of solutions of \eqref{eq:h-floer} converges to $u_\infty$ on compact subsets, and $T \subset S$ is any finite cylinder, then $\pi(u_k|T)$, $k \gg 0$, is a solution of \eqref{eq:cyl-equation}. Taking $T$ to be sufficiently long yields a contradiction to Lemma \ref{th:length}. Hence, we have now shown the following:

\begin{lemma} \label{th:length-trick}
All solutions of \eqref{eq:h-floer} are contained in a compact subset of $E \setminus \partial E$. \qed
\end{lemma}

Given that, it is straightforward to define the associated Hamiltonian Floer complex $\mathit{CF}^*(E,1)$, whose cohomology we denote by $\mathit{HF}^*(E,1)$ (the number $1$ records the rotation number of the holonomy of the connection on the circle).

Let $S$ be a disc with $(d+1)$ boundary punctures and an interior puncture, equipped with some $(A,\tilde{\lambda}) \in \scrP(S,\Sigma)$. As before, we choose $J = (J_z)$, $K$, and Lagrangian boundary conditions $L_z$. On the cylindrical end, we want the analogue of \eqref{eq:match-1}--\eqref{eq:match-2} to hold, where $J_{d+1}$ and $H_{d+1}$ are such that they can be used to define the Hamiltonian Floer complex. We consider solutions of the associated equation \eqref{eq:main-cr}, \eqref{eq:with-limits}, where the last limit is $x_{d+1} \in \scrC(H_{d+1})$. The required compactness properties follow by combining the previous arguments with that from Lemma \ref{th:length-trick}. 

To define the closed open-string map, one considers boundary conditions which are locally constant along $\partial S$, and uses data on the strip-like ends dictated by the previous definition of the Fukaya category. Moreover, one varies over all Riemann surfaces $S$. The outcome is a collection of maps
\begin{equation}
\mathit{CF}^*(E,1) \otimes \mathit{CF}^*(L_{d-1},L_d) \otimes \cdots \otimes \mathit{CF}^*(L_0,L_1) \longrightarrow \mathit{CF}^{*-d}(L_0,L_d),
\end{equation}
for objects $(L_0,\dots,L_d)$, which together form a chain map from $\mathit{CF}^*(E,1)$ into the standard Hochschild cochain complex $\mathit{CC}^*(\scrA,\scrA)$. This construction is the same as in \cite{seidel02}, \cite[Section 3.8]{fooo}, or \cite[Section 5.4]{ganatra13}, except that on each surface $S$, the choice of angular parametrization of the tubular end $\epsilon_{d+1}$ is tied to constructing the flat connection, in order to utilize the contractibility statement from Corollary \ref{th:last-contractible}. 

\begin{remark}
To expand on the last sentence, note that in standard setups (of the closed-open string map for compact Lagrangian submanifolds, or in the setting of wrapped Fukaya category), rotating the parametrization of the closed string end yields another degree of freedom. Using that degree of freedom yields another map, which turns out to be the composition of the closed-open string map with the BV (loop rotation) operator on Hamiltonian Floer cohomology. In our context, where $\mathit{HF}^*(E,1)$ is defined using a Hamiltonian that is fundamentally (because of the desired rotation number, and not just for technical reasons of transversality) time-dependent, there is no BV operator.
\end{remark}


\begin{thebibliography}{10}
\bibitem{abbas14}
C.~Abbas.
\newblock {\em An introduction to compactness results in {S}ymplectic {F}ield
  {T}heory}.
\newblock Springer, 2014.

\bibitem{abouzaid10}
M.~Abouzaid.
\newblock A geometric criterion for generating the {F}ukaya category.
\newblock {\em Publ. Math. IHES}, 112:191--240, 2010.

\bibitem{abouzaid-ganatra14}
M.~Abouzaid and S.~Ganatra.
\newblock Generating {F}ukaya categories of {L}andau-{G}inzburg models.
\newblock Talks at the {MIT} workshop on {L}efschetz fibrations, 2015.

\bibitem{fooo}
K.~Fukaya, Y.-G. Oh, H.~Ohta, and K.~Ono.
\newblock {\em Lagrangian intersection {F}loer theory - anomaly and
  obstruction}.
\newblock Amer. Math. Soc., 2010.

\bibitem{ganatra13}
S.~Ganatra.
\newblock Symplectic cohomology and duality for the wrapped {F}ukaya category.
\newblock Preprint arXiv:1304.7312, 2013.

\bibitem{perutz10}
T.~Perutz.
\newblock Talk at the {AIM} workshop {\em Cyclic homology and symplectic
  topology}, 2009.

\bibitem{ganatra-pardon-shende17}
J.~Pardon, S.~Ganatra and V.~Shende.
\newblock Covariantly functorial {F}loer theory on {L}iouville sectors.
\newblock Preprint arXIv:1706.03152, 2017.

\bibitem{seidel00b}
P.~Seidel.
\newblock More about vanishing cycles and mutation.
\newblock In {\em {S}ymplectic {G}eometry and {M}irror {S}ymmetry ({P}roceedings of the 4th
  {KIAS} Annual International Conference)}, pages 429--465. World Scientific,
  2001.

\bibitem{seidel00}
P.~Seidel.
\newblock Vanishing cycles and mutation.
\newblock In {\em Proceedings of the 3rd European Congress of Mathematics
  (Barcelona, 2000)}, volume~2, pages 65--85. Birkh{\"a}user, 2001.

\bibitem{seidel02}
P.~Seidel.
\newblock Fukaya categories and deformations.
\newblock In {\em Proceedings of the International Congress of Mathematicians
  (Beijing)}, volume~2, pages 351--360. Higher Ed. Press, 2002.

\bibitem{seidel14b}
P.~Seidel.
\newblock {F}ukaya {$A_\infty$}-categories associated to {L}efschetz
  fibrations. {II}.
\newblock In {\em Algebra, {G}eometry and {P}hysics in the 21st
  {C}entury ({K}ontsevich {F}estschrift)}, pages 295--364. Birkh{\"a}user, 2017.

\bibitem{seidel17}
P.~Seidel.
\newblock Fukaya {$A_\infty$}-structures associated to {L}efschetz fibrations.
  {IV}.
\newblock In {\em Proceedings of the 2017 Georgia International Topology Conference.} Amer.\ Math.\ Soc., to appear.

\bibitem{sylvan16}
Z.~Sylvan.
\newblock On partially wrapped {F}ukaya categories.
\newblock Preprint arXiv:1604.02540, 2016.

\end{thebibliography}

\end{document}